\documentclass{article}

\usepackage{amsfonts}
\usepackage{amssymb}
\usepackage{amsmath}
\usepackage{graphicx}
\usepackage{stmaryrd}  
\usepackage{tikz}
\usetikzlibrary{shapes,arrows,trees, calc}
\tikzstyle{smallball blue}=[shading=radial,outer color=blue,inner color=white,line width=0.2pt]
\tikzstyle{smallball red}=[shading=radial,outer color=red,inner color=white,line width=0.2pt]

\usepackage[lined,boxruled]{algorithm2e} 
\usepackage[skip=5pt,font=scriptsize]{caption} 

\usepackage{rotfig} 

\usepackage{arydshln}

\usepackage{color} 
   \definecolor{cites}{rgb}{0.50 , 0.00 , 0.00}  
   \definecolor{urls} {rgb}{0.00 , 0.00 , 0.50}  
   \definecolor{links}{rgb}{0.00 , 0.00 , 0.50}   

\usepackage{hyperref} 
\hypersetup{  
      colorlinks=true,   
      citecolor=cites,   
      urlcolor=urls,     
      linkcolor=links,   
      pdftitle={Recycling Givens rotations for the efficient approximation of pseudospectra of band-dominated operators},  
      pdfauthor={Marko Lindner, Torge Schmidt}, 
      pdfstartview=FitH,       
      bookmarksopen=false      
}

\newcommand\C{{\mathbb C}}
\newcommand\D{{\mathbb D}}

\newcommand\N{{\mathbb N}} 
\renewcommand\O{{\mathcal O}}

\newcommand\T{{\mathbb T}}
\newcommand\R{{\mathbb R}}
\newcommand\W{{\mathcal W}}
\newcommand\Z{{\mathbb Z}}
\newcommand\zero{\textbf{0}}

\newcommand\eps{{\varepsilon}}
\newcommand\spec{{\rm spec}\,}  
\newcommand\specn{{\rm spec}}   
\newcommand\speps{{\rm spec}_\eps}

\newcommand\diam{{\rm diam\,}}
\newcommand\diag{{\rm diag}}

\newcommand\supp{{\rm supp\,}}
\newcommand{\abs}[1]{\left|#1\right|}
\newcommand\norm[1]{\left\|#1\right\|}

\newcommand\BDO{{\rm BDO}}
\newcommand\BO{{\rm BO}}

\newcommand\lW{{\llbracket}}
\newcommand\rW{{\rrbracket}}

\newcommand{\Rr}{{\color{gray}\Rc}}
\newcommand{\rr}{{\color{gray}\rc}}

\newtheorem{theorem}{Theorem}[section]

\newenvironment{example}
 {\par\noindent\refstepcounter{theorem}{\bf Example \thetheorem}}
 {\raisebox{1mm}{\framebox{}}\pagebreak[2]}

\newcommand\Proofend{\rule{2mm}{2mm}}

\newenvironment{proof}
 {\par\noindent{\bf Proof.}}
 {\Proofend\pagebreak[2]}

\numberwithin{figure}{section}  

\begin{document}
\title{\bf Recycling Givens rotations for the\\efficient approximation of pseudospectra\\of band-dominated operators}
\author{{\sc Marko Lindner}\footnote{Email: {\tt lindner@tuhh.de}}\quad and\quad {\sc Torge Schmidt}\footnote{Email: {\tt torge.schmidt@tuhh.de}}}
\date{\today}
\maketitle
\begin{quote}
\renewcommand{\baselinestretch}{1.0}
\footnotesize {\sc Abstract.}
We study spectra and pseudospectra of certain bounded linear operators on $\ell^2(\Z)$. The operators are generally non-normal, and their matrix representation has a characteristic off-diagonal decay. Based on a result of Chandler-Wilde, Chonchaiya and Lindner for tridiagonal infinite matrices, we demonstrate an efficient algorithm for the computation of upper and lower bounds on the pseudospectrum of operators that are merely norm limits of band matrices -- the so-called band-dominated operators. After approximation by a band matrix and fixing a parameter $n\in\N$, one looks at $n$ consecutive columns $\{k+1,...,k+n\}$, $k\in\Z$, of the corresponding matrix and computes the smallest singular value of that section via QR factorization. We here propose a QR factorization by a sequence of Givens rotations in such a way that a large part of the computation can be reused for the factorization of the next submatrix -- when $k$ is replaced by $k+1$. The computational cost for the next factorization(s) is $\O(nd)$ as 
opposed to a naive implementation with $\O(nd^2)$, where $d$ is the bandwidth. So our algorithm pays off for large bands, which is attractive when approximating band-dominated operators with a full (i.e.~not banded) matrix.
\end{quote}

{\it Keywords:} 

\section{Introduction, Notations, and Main Results} \label{sec:intro}
{\bf Band-dominated operators.} 
We study bounded linear operators on the space $\ell^2:=\ell^2(\Z)$ of square-summable bi-infinite complex sequences $x=(x_k)_{k\in\Z}$ with $\|x\|=\sqrt{\sum_{k\in\Z} |x_k|^2}<\infty$. Each linear operator $A$ on $\ell^2$ acts via matrix-vector multiplication with a bi-infinite matrix $(a_{ij})_{i,j\in\Z}$ -- and vice versa. We say that $A$ is a {\sl band operator} if its matrix $(a_{ij})$ is banded (i.e.~supported on only finitely many diagonals) and has uniformly bounded entries, so that $A$ is a bounded linear operator. In that case, $d:=\max\{|i-j|:a_{ij}\ne 0\}$ is called the {\sl bandwidth} of $A$. Moreover,
$A$ is called a {\sl band-dominated operator} if it is the limit, in the induced operator norm on $\ell^2$, of a sequence of band operators; in particular it is a bounded operator, too, and its matrix entries decay with their distance from the main diagonal.

{\bf Pseudospectra.}
Because the spectrum of a non-normal operator $A$ can be highly unstable under small perturbations of $A$, one is interested in the so-called $\eps$-pseudospectrum of $A$, that is,
\[
\speps A\ :=\ \{\lambda\in\C:\|(A-\lambda I)^{-1}\|>1/\eps\}\ =\ \bigcup_{\|T\|<\eps}\spec(A+T),\qquad\eps>0.
\]
Here we agree upon writing $\|B^{-1}\|=\infty$ if $B$ is not invertible. The second equality sign (see e.g. \cite{TrefEmbBook}) shows that $\speps A$ measures the sensitivity of $\spec A$ w.r.t.~additive perturbations of $A$ of norm $<\eps$. For normal operators $A$, $\speps A$ is the $\eps$-neigbourhood of $\spec A$; otherwise it is generally larger (but never smaller). The interest in pseudospectra has been increasing over the last two decades. See \cite{TrefEmbBook} for many more reasons to study pseudospectra and for more references. 

{\bf The lower norm.}
As a counterpart to the operator norm $\|A\|=\sup_{\|x\|=1}\|Ax\|$, we look at the quantity
\[
\nu(A)\ :=\ \inf_{\|x\|=1}\|Ax\|,
\]
that is sometimes (by abuse of notation) called the {\sl lower norm} of $A$. While $\|A\|$ is the largest singular value of $A$, $\nu(A)$ is the smallest -- provided maximum/minimum exist, such as in the case of finite matrices. It is well-known (see e.g. \cite[p.69f]{LiBook}) that $\nu(A)>0$ holds iff $A$ is injective and has a closed image; moreover, the equality
\[
\|A^{-1}\|\ =\ 1/\min(\nu(A),\nu(A^*))
\]
holds with $1/0:=\infty$ indicating non-invertibility of $A$. In particular, $A$ is invertible iff $\nu(A)$ and $\nu(A^*)$ are both nonzero, in which case they coincide. Together with the definition of $\speps A$ it follows that
\begin{equation} \label{eq:spepsnu}
\speps A\ =\ \{\lambda\in\C:\min\!\big(\nu(A-\lambda I),\nu((A-\lambda I)^*)\big)<\eps\}.
\end{equation}

{\bf Approximating the lower norm of band-dominated operators.}
For $x\in\ell^{2}$, we denote its support by $\supp x:=\{j\in\Z:x_j\ne 0\}$, and we say that a bounded set $J\subset\Z$ has diameter $\diam J:=\max\{|i-j|:i,j\in J\}$.
One of the main observations of \cite{CW.Heng.ML:UpperBounds} (also see \cite{HengPhD,LiSei:BigQuest}) is that the lower norm\footnote{A symmetric result holds for the norm, $\|A\|$, see Proposition 3.4 and inequality (ONL) in \cite{HagLiSei}.} of a band-dominated operator $A$ can be realized, up to a given $\delta>0$, by a unit element $x\in\ell^{2}$ with bounded support, say of diameter less than $n\in\N$ (dependent on $\delta$, of course). So one has
\begin{equation} \label{eq:nusqueeze}
\nu(A)\ \le\ \|Ax\|\ \le\ \nu(A)+\delta
\end{equation}
for a particular $x\in\ell^{2}$ with $\|x\|=1$ and $\diam(\supp x)<n$. If $\supp x$ were known to be contained in the discrete interval $J^n_k:=\{k+1,...,k+n\}$ with a given $k\in\Z$, then the optimal term $\|Ax\|$ in \eqref{eq:nusqueeze} could be practically computed as the lower norm / smallest singular value of the restriction of $A$ to $\ell^{2}(J^n_k)$. Since $\diam(\supp x)<n$, the support must be contained in some interval $J^n_k$ with $k\in\Z$. Unfortunately, this $k$ is in general not known.  It ``remains'' to look at -- and minimize over -- all $k\in\Z$: 
\begin{equation} \label{eq:nusqueezek}
\nu(A)\ \le\ \inf_{k\in \Z}\nu(A|_{\ell^{2}(J^n_k)})\ \le\ \nu(A)+\delta
\end{equation}
If $A$ is a band operator then $A|_{\ell^{2}(J^n_k)}$ corresponds to a finite rectangular matrix (containing columns $k+1,...,k+n$ of the infinite matrix, truncated to their joint support that is finite -- due to the band structure), so that the smallest singular value, $\nu(A|_{\ell^{2}(J^n_k)})$, can be computed effectively.
However, consideration of all $k\in\Z$ is, in general, of course practically impossible -- unless the set of all restrictions $\{A|_{\ell^{2}(J^n_k)}:k\in\Z\}$ is finite, 
e.g. when $A$ is eventually periodic or otherwise structured.

It is clear that the size $n$ has to be increased in order to decrease the error $\delta$ in \eqref{eq:nusqueeze} and \eqref{eq:nusqueezek}. The analysis in \cite{CW.Heng.ML:UpperBounds} (also see \S 3 and 4 in \cite{HengPhD}) shows, for the particular case of tridiagonal (bandwidth $d=1$) bi-infinite matrices $(a_{ij})_{i,j\in\Z}$, that $\delta$ is of the order $1/n$; more precisely,
\begin{equation}\label{eq:epsn}
\delta\ \le\ 2\left(\sup_{j\in\Z}|a_{j+1,j}|+\sup_{j\in\Z}|a_{j-1,j}|\right)\sin\frac\pi{2n+2}\ \in\ \O\left(\frac 1n\right),
\end{equation}
The constant turns out to be optimal. We make use of that result by two simple steps of reduction: 
\begin{enumerate}
\item[$(i)$] Given an accuracy $\eta>0$, approximate our band-dominated operator $A$ (with a generally full matrix) by a band operator $B$ with $\|A-B\|\le\eta$ and use the contractivity of $\nu(\cdot)$, so that
\begin{equation}\label{eq:nu.contr}
|\nu(A)-\nu(B)|\le\|A-B\|\le\eta,\quad\textrm{as well as}\quad |\nu(A^*)-\nu(B^*)|\le\|A^*-B^*\|\le\eta.
\end{equation}

\item[$(ii)$] Use that the matrix of the band operator $B$ is block-tridiagonal (with block size equal to the band width of $B$, see Figure 1) and that the results of \cite{CW.Heng.ML:UpperBounds,HengPhD} even apply to tridiagonal matrices with operator entries\footnote{In that case, $|a_{j+1,j}|$ and $|a_{j-1,j}|$ in \eqref{eq:epsn} are interpreted as operator norms.} -- hence to block-tridiagonal matrices.
\end{enumerate}

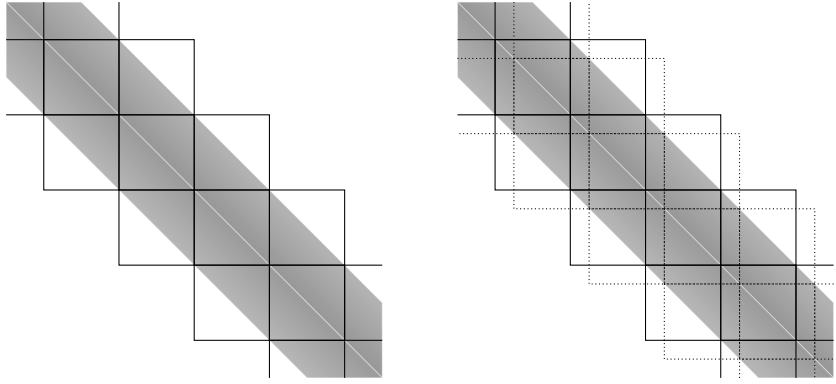
\begin{figure}[h]
\begin{center}
\begin{tikzpicture}[every node/.style={minimum size=1cm}]

    \begin{scope}[scale=0.5]
    \shadedraw[left color=white, right color=white, middle color=black, shading angle=-45,draw=white, opacity=0.4] (0,0) -- (10,-10)  -- (10,-8) -- (2,0) --cycle;
    \shadedraw[left color=white, right color=white, middle color=black, shading angle=-45,draw=white, opacity=0.4] (0,0) -- (10,-10) -- (8,-10) -- (0,-2) --cycle;
 
    \draw (0,-1) --(1,-1) --(1,-3) -- (0,-3);
    \draw (1,0) -- (1,-1) -- (3,-1)--(3,0);
    \draw (1, -1) rectangle (3,-3);
    \draw (1,-3) rectangle (3,-5);
    
    \draw (3, -1) rectangle (5,-3);
    \draw (3,-3) rectangle (5,-5);
    \draw (3, -5) rectangle (5,-7);
    
    \draw (5, -3) rectangle (7,-5);
    \draw (5,-5) rectangle (7,-7);
    \draw (5, -7) rectangle (7,-9);
    
    \draw (7, -5) rectangle (9,-7);
    \draw (7,-7) rectangle (9,-9);
    
    \draw (10,-7) -- (9,-7) -- (9,-9) -- (10,-9);
    \draw (7,-10) -- (7,-9) -- (9,-9) -- (9,-10);

    \end{scope}
      
    \begin{scope}[scale=0.5,xshift=12cm]
     \shadedraw[left color=white, right color=white, middle color=black, shading angle=-45,draw=white, opacity=0.4] (0,0) -- (10,-10)  -- (10,-8) -- (2,0) --cycle;
    \shadedraw[left color=white, right color=white, middle color=black, shading angle=-45,draw=white, opacity=0.4] (0,0) -- (10,-10) -- (8,-10) -- (0,-2) --cycle;
 
    \draw (0,-1) --(1,-1) --(1,-3) -- (0,-3);
    \draw (1,0) -- (1,-1) -- (3,-1)--(3,0);
    \draw (1, -1) rectangle (3,-3);
    \draw (1,-3) rectangle (3,-5);
    
    \draw (3, -1) rectangle (5,-3);
    \draw (3,-3) rectangle (5,-5);
    \draw (3, -5) rectangle (5,-7);
    
    \draw (5, -3) rectangle (7,-5);
    \draw (5,-5) rectangle (7,-7);
    \draw (5, -7) rectangle (7,-9);
    
    \draw (7, -5) rectangle (9,-7);
    \draw (7,-7) rectangle (9,-9);
    
    \draw (10,-7) -- (9,-7) -- (9,-9) -- (10,-9);
    \draw (7,-10) -- (7,-9) -- (9,-9) -- (9,-10);

    \draw[densely dotted] (0,-1.5) --(1.5,-1.5) --(1.5,-3.5) -- (0,-3.5);
    \draw[densely dotted] (1.5,0) -- (1.5,-1.5) -- (3.5,-1.5)--(3.5,0);
    \draw[densely dotted] (1.5, -1.5) rectangle (3.5,-3.5);
    \draw[densely dotted] (1.5,-3.5) rectangle (3.5,-5.5);
    
    \draw[densely dotted] (3.5, -1.5) rectangle (5.5,-3.5);
    \draw[densely dotted] (3.5,-3.5) rectangle (5.5,-5.5);
    \draw[densely dotted] (3.5, -5.5) rectangle (5.5,-7.5);
    
    \draw[densely dotted] (5.5, -3.5) rectangle (7.5,-5.5);
    \draw[densely dotted] (5.5,-5.5) rectangle (7.5,-7.5);
    \draw[densely dotted] (5.5, -7.5) rectangle (7.5,-9.5);
    
    \draw[densely dotted] (7.5, -5.5) rectangle (9.5,-7.5);
    \draw[densely dotted] (7.5,-7.5) rectangle (9.5,-9.5);
    
    \draw[densely dotted] (10,-7.5) -- (9.5,-7.5) -- (9.5,-9.5) -- (10,-9.5);
    \draw[densely dotted] (7.5,-10) -- (7.5,-9.5) -- (9.5,-9.5) -- (9.5,-10);
    \end{scope}
\end{tikzpicture}
\caption{Left: A banded matrix (support shown in gray) is  turned into block-tridiagonal form with blocks of according size. Right: The dotted blocks equally do the job of turning the banded matrix into block-tridiagonal form. There are $b$
different ways of positioning a $b\times b$ grid along the main diagonal. Two of them are depicted here (solid and dotted lines).}
\label{fig:1.1}
\end{center}
\end{figure}

We discuss further details of steps $(i)$ and $(ii)$ in Section \ref{sec:red}.

{\bf Approximating pseudospectra of band-dominated operators.}
From \eqref{eq:spepsnu} and the above approximations and bounds on the lower norm we conclude approximations and bounds on the pseudospectrum:

Inequality \eqref{eq:nusqueezek} and its counterpart for the adjoint, $A^*$, lead to
\[
\min(\nu(A),\nu(A^*))\ \le\ \inf_{k\in \Z}\min\!\big(\nu(A|_{\ell^{2}(J^n_k)}),\nu(A^*|_{\ell^{2}(J^n_k)})\big)\ \le\ \min(\nu(A),\nu(A^*))+\delta,
\]
from which we conclude the implications
\begin{align*}
\inf_{k\in \Z}\min\!\big(\nu(A|_{\ell^{2}(J^n_k)}),\nu(A^*|_{\ell^{2}(J^n_k)})\big)<\eps
\quad&\Rightarrow\quad
\min(\nu(A),\nu(A^*))<\eps\\
&\Rightarrow\quad
\inf_{k\in \Z}\min\!\big(\nu(A|_{\ell^{2}(J^n_k)}),\nu(A^*|_{\ell^{2}(J^n_k)})\big)<\eps+\delta
\end{align*}
for all $\eps>0$, and consequently
\begin{equation}\label{eq:speps.sandwich}
\Gamma^n_\eps(A)\ \subset\ \speps A\ \subset\ \Gamma^n_{\eps+\delta}(A),
\end{equation}
where
\begin{equation}\label{eq:defGamma}
\Gamma^n_\eps(A)\ :=\ \bigcup_{k\in\Z}\left\{\lambda\in\C:\min\!\big(\nu((A-\lambda I)|_{\ell^{2}(J^n_k)}),\nu((A-\lambda I)^*|_{\ell^{2}(J^n_k)})\big)<\eps \right\}.
\end{equation}
Concerning the approximation step $(i)$ above, by \eqref{eq:nu.contr}, we have the implications
\[
\nu(B)<\eps-\eta
\qquad\Rightarrow\qquad
\nu(A)<\eps
\qquad\Rightarrow\qquad
\nu(B)<\eps+\eta,
\]
and the same holds for the adjoints. Subtracting $\lambda I$ from $A$ and $B$ and using \eqref{eq:spepsnu}, this shows that
\begin{equation}\label{eq:spepsAB}
\specn_{\eps-\eta}B\ \subset\ \speps A\ \subset\ \specn_{\eps+\eta}B,\qquad 0<\eta<\eps,
\end{equation}
so that upper and lower bounds on certain pseudospectra of $B$ yield bounds on $\speps A$. Moreover, the inclusions \eqref{eq:spepsAB} are as tight
as desired (in the Hausdorff distance) by sending $\eta\to 0$. 

{\bf Existing results.}
The probably most natural idea to approximate $\speps A$ is to look at the pseudospectra $\speps A_n$ of the finite sections $A_n=(a_{ij})_{i,j=-n}^n$
of $A=(a_{ij})_{i,j\in\Z}$ as $n\to\infty$. In some rare cases (Toeplitz operators \cite{ReichelTref,BoeSi2}, random Jacobi operators
\cite{CWLi2016:Coburn}), the sets $\speps A_n$ indeed converge to $\speps A$ w.r.t.~the Hausdorff distance -- but in general, the sequence
$\speps A_n$ does not converge at all; its cluster points usually contain $\speps A$ but also further points (see e.g. \cite{SeSi:FSBDO}, one speaks of spectral pollution).
Even in a simple selfadjoint example such as $A=\diag(...,B,B,B,...)$ with $B={0~1\choose 1~0}$, one has\footnote{The $\eps$-pseudospectra are
the $\eps$-neighbourhoods of the spectra in this selfadjoint example.} $\spec A=\{-1,1\}$, while $\spec A_n$ repeatedly switches between $\{-1,1\}$
and $\{-1,0,1\}$ as $n$ grows. As an alternative that is somewhere between spectra and pseudospectra, \cite{Hansen:nPseudo,Seidel:Neps} study
so-called $(N,\eps)$-pseudospectra, where $2^N$-th powers of the resolvent and of $1/\eps$ are compared to each other. In \cite{HansenSeidel}
the lower norms of rectangular submatrices are suggested for the approximation of the spectrum and the $(N,\eps)$-pseudospectrum. Needless to say,
there is a large amount of literature on the selfadjoint case (see e.g. \cite{Boulton,DaviesPlum} and the references therein).

One major problem in approximating $\Z$ by the intervals $\{-n,...,n\}$ is (besides the potential of spectral pollution) that generally, 
huge values of $n$ are required to capture spectral properties of $A$ properly. (Think of an infinite diagonal matrix with distinguished 
entries in very remote locations.) From a computational perspective, such huge sections $\{-n,...,n\}$ are too expensive. The approach of 
\cite{CW.Heng.ML:UpperBounds} (also see \S 3 and 4 in \cite{HengPhD}) -- that is very much in the spirit of Gershgorin and that we adopt here -- replaces $\{-n,...,n\}$ with 
$n\to\infty$ by the family of intervals $J_k^n=\{k+1,...,k+n\}$ for all $k\in\Z$ but with $n$ of moderate size. The price that is obviously 
paid is the infinite amount of positions $k$ that one has to look at, so that a certain structural simplicity of the infinite matrix is required to make the approach practically 
feasible. The other major plus of the \cite{CW.Heng.ML:UpperBounds} approach is that it comes with sharp and explicit bounds \eqref{eq:epsn} 
on the accuracy of the approximation \eqref{eq:speps.sandwich}, while working for the general non-normal case.

{\bf What is new here?}
The tridiagonal results and the ideas of transferring them to band-dominated operators via $(i)$ and $(ii)$ are from
\cite{CW.Heng.ML:UpperBounds}, therefore not new. But there are two degrees of freedom in the choice of the blocks in step $(ii)$: 
Firstly, the size of the blocks, say $b\in\N$, could be any number greater than or equal to the bandwidth $d\in\N$. Secondly, once this size $b$ 
is fixed, there are $b$ different choices for the position of the blocks inside the infinite matrix (see Figure \ref{fig:1.1}). 

We play with that second degree of freedom, arguing that there is usually no best choice (in terms of sharpness of \eqref{eq:speps.sandwich})
of block positioning, and instead we consider all $b$ possibilities, thereby improving sharpness of the bounds on $\speps A$. (We take the union 
of the $b$ different lower bounds and the intersection of the $b$ upper bounds.) Naively implemented, this increases the computational cost by
the factor $b$. However, we present an algorithm that compensates for this increase by reusing much of the effort that was put into the computation
of $\nu(A|_{\ell^2(J^n_k)})$ for the computation of $\nu(A|_{\ell^2(J^n_{k+1})})$. This is possible due to the large overlap between the two
matrices $A|_{\ell^2(J^n_k)}$ and $A|_{\ell^2(J^n_{k+1})}$. We cannot see a similar idea to work for the $b$-sized step from $\nu(A|_{\ell^2(J^n_k)})$
to $\nu(A|_{\ell^2(J^n_{k+b})})$ in the block matrix, though.

In a nutshell, the smallest singular value of $A|_{\ell^2(J^n_k)}$ coincides with that of the upper triangular matrix\footnote{For the computation of the smallest singular value of $R_k$, one can use an inverse Lanczos method.} $R_k$ from the factorization $A|_{\ell^2(J^n_k)}=Q_kR_k$ with a unitary $Q_k$ that results from a sequence of Givens rotations. The key idea is now to rearrange and reuse most of these Givens rotations 
for the next step when $k$ is replaced by $k+1$. With this algorithm, the complexity of the computation of $\nu(A|_{\ell^2(J^n_{k+1})})=\nu(R_{k+1})$ decreases from $\O(nd^2)$ to just $\O(nd)$, thereby compensating for the increase by a factor of $b\approx d$ that was mentioned above. The same recycling idea and the same complexity then also apply to the computation of $\nu(A|_{\ell^2(J^n_{k+2})}), \nu(A|_{\ell^2(J^n_{k+3})})$, etc.

{\bf Contents of the paper.}
In Section \ref{sec:red} we show the details of both reduction steps $(i)$ and $(ii)$. The heart of the paper is Section \ref{sec:algorithm}, where we present the algorithm for the computation of $\nu(A|_{\ell^2(J^n_{k+1})})$ from $\nu(A|_{\ell^2(J^n_k)})$ by appropriately reordering Givens rotations.
In Section \ref{sec:applications} we illustrate our results in two examples with non-trivial pseudospectra. Moreover, we compare the efficiency of our algorithm with the standard QR decomposition in each step.

\section{From band-dominated to tridiagonal operators} \label{sec:red}
Recall that we call an operator $A$ on $\ell^2$ band-dominated if it is the limit, in the operator norm, of a sequence of band operators (which are bounded operators with a banded matrix representation). Let us denote the sets of all band and all band-dominated operators on $\ell^2$ by $\BO$ and $\BDO$, respectively. We make use of the results from \cite{CW.Heng.ML:UpperBounds,HengPhD} for tridiagonal operators by two steps of reduction:

\subsection{Step $(i)$: From band-dominated to banded}\label{ssec:banddominatedtobanded}
Let $A\in\BDO$ and $\eta>0$ be given. There are different approaches of constructing a band operator $B$ with $\|A-B\|\le\eta$, leading to \eqref{eq:nu.contr} and \eqref{eq:spepsAB}:

{\bf Case 1.} If $A$ is in the so-called Wiener algebra, the problem is simple. To explain this, let $(b_{ij})_{i,j\in\Z}$ be the matrix representation of some $B\in\BO$ and let $d_k:=(b_{j+k,j})_{j\in\Z}$ be its $k$-th diagonal, where $k\in\Z$. Then
\[
B\ =\ \sum_{k\in\Z} M_{d_k}V_k,
\]
where $M_f$ refers to the operator on $\ell^2$ of entrywise multiplication with a sequence $f\in\ell^\infty$ and $V_k$ is the forward shift on $\ell^2$ by $k$ positions. (Note that the sum is actually finite, by $B\in\BO$.) It now follows that
\begin{equation} \label{eq:Wnorm}
\|B\|\ =\ \left\|\sum_{k\in\Z} M_{d_k}V_k\right\|\ \le\ \sum_{k\in\Z} \|M_{d_k}\|\|V_k\|\ =\ \sum_{k\in\Z} \|d_k\|_\infty\ =:\ \lW B\rW.
\end{equation}
The new expression $\lW\cdot\rW$ indeed defines a norm on $\BO$. The completion of $\BO$ with respect to $\lW\cdot\rW$ is the so-called {\sl Wiener algebra} $\W$. By \eqref{eq:Wnorm}, $\W$ is contained in the completion of $\BO$ w.r.t. $\|\cdot\|$, that is $\BDO$. Moreover, $(\W,\lW\cdot\rW)$ is a Banach algebra (see \S1.6.8 in \cite{Kurbatov} or \S3.7.3 in \cite{Li:Habil}).

So if $A\in\W\subset\BDO$ and $d_k$ refers to its $k$-th diagonal for all $k\in\Z$, then
\begin{equation}\label{eq:Wapprox}
B_n\ :=\ \sum_{k=-n}^n M_{d_k}V_k\ \in\BO
\end{equation}
is the desired approximation of $A$ if $n\in\N$ is chosen large enough for
\begin{equation}
\|A-B_n\|\ \le\ \lW A-B_n\rW\ =\ \sum_{|k|>n}\|d_k\|_\infty\ \le\ \eta.\label{eq:Werror}
\end{equation}
Such an $n$ exists since $\sum_{n\in\Z}\|d_k\|_\infty<\infty$, by $A\in\W$.

{\bf Case 2.} If $A\in\BDO\setminus\W$, the simple approach \eqref{eq:Wapprox} of restriction to a finite subset of diagonals need not lead to a sequence $B_n$ that converges to $A$ in the operator norm. A simple example is shown in Remark 1.40 of \cite{LiBook}. The example relies on the fact that, for a continuous $2\pi$-periodic function $f$ on $\R$, the partial sums of the Fourier series need not converge uniformly to $f$. This is repaired by looking at Fejer-Cesaro means instead, and the same trick works for the approximation of band-dominated operators:
\begin{equation}
C_n\ :=\ \frac{B_0+...+B_n}{n+1}\ =\ \sum_{k=-n}^n \left(1-\frac{|k|}{n+1}\right)M_{d_k}V_k\ \in\BO
\end{equation}
with $B_n$ from \eqref{eq:Wapprox} can be shown to converge to $A$ in the operator norm as $n\to\infty$, see e.g. the proof of the implication $(e)\Rightarrow(a)$ in Theorem 2.1.6 of \cite{RaRoSiBook}.

Another way to explicitly approximate $A\in\BDO$ by band operators is shown in (1.18) of \cite{SeidelPhD}.

\subsection{Step $(ii)$: From banded to tridiagonal}\label{ssec:bandedtotridiagonal}
Now we can assume $A\in\BO$. Let $d$ denote its bandwidth. The idea is captured by Figure \ref{fig:1.1} above: $A$ can be expressed as a block-tridiagonal matrix with block size $b\ge d$. Besides the choice of $b$, there is another degree of freedom in this identification. If the blocks are centered on the main diagonal, there are still $b$ different positions at which to start, see Figure \ref{fig:1.1}. 

Precisely, each block is of the form 
\begin{equation}\label{eq:blocks}
\left(
\begin{array}{ccc}
a_{i+1,j+1}&\cdots&a_{i+1,j+b}\\
\vdots& &\vdots\\
a_{i+b,j+1}&\cdots&a_{i+b,j+b}
\end{array}
\right)\in\C^{b\times b}
\quad\textrm{with}\quad
i,j\in c+b\Z:=\{c+bz:z\in\Z\},
\end{equation}
where $c\in\{0,...,b-1\}$ is this second degree of freedom. This leads to $b$ different ways (one for each choice of the offset $c$) of turning $A$ into a tridiagonal matrix.

For the moment, fix one choice of $c\in\{0,...,b-1\}$. To apply the results on the block-tridiagonal matrix behind $A$, we have to adjust the intervals $J^n_k:=\{k+1,...,k+n\}$ (of matrix columns under current investigation in \eqref{eq:nusqueezek}) with the blocks. Therefore, we restrict ourselves to positions $k\in c+b\Z$ and to interval lengths $n=Nb$, where $N\in\N$ is the number of blocks to be considered in $J^n_k$.

Now we slightly modify \eqref{eq:defGamma} to
\begin{equation}\label{eq:defGammaM}
\Gamma^{n,M}_\eps(A)\ :=\ \bigcup_{k\in M}\left\{\lambda\in\C:\min\!\big(\nu((A-\lambda I)|_{\ell^{2}(J^n_k)}),\nu((A-\lambda I)^*|_{\ell^{2}(J^n_k)})\big)<\eps \right\}
\end{equation}
for any set $M\subset\Z$, where our particular interest is in sets of the form $M=c+b\Z$. Assuming $b$ as given and fixed, we abbreviate $\Gamma^{n,c+b\Z}_\eps(A)=:\Gamma^{n,c}_\eps(A)$.

Because each offset $c\in\{0,...,b-1\}$ yields a tridiagonal representation of $A$, we get from \eqref{eq:speps.sandwich} that all inclusions
\begin{equation} \label{eq:Gammas}
\left.
\begin{array}{rcl}
\Gamma^{n,0}_\eps(A)\ \subset& \speps A& \subset\ \Gamma^{n,0}_{\eps+\delta}(A)\\
\Gamma^{n,1}_\eps(A)\ \subset& \speps A& \subset\ \Gamma^{n,1}_{\eps+\delta}(A)\\
&\vdots\\
\Gamma^{n,b-1}_\eps(A)\ \subset& \speps A& \subset\ \Gamma^{n,b-1}_{\eps+\delta}(A)
\end{array}
\qquad\right\}
\end{equation}
hold. Here, by evaluating \eqref{eq:epsn} for the block tridiagonal matrix,
\begin{equation}\label{eq:epsn_block}
\delta\ \le\ 2\left(\sup_{l\in\Z}\|A_{l+1,l}\|+\sup_{l\in\Z}\|A_{l-1,l}\|\right)\sin\frac\pi{2N+2}\ \in\ \O\left(\frac 1N\right)\ =\ \O\left(\frac 1n\right),
\end{equation}
where we denote the block \eqref{eq:blocks} by $A_{kl}$ if $i=c+bk$ and $j=c+bl$ with $k,l\in\Z$.

Taking unions on the left and intersections on the right of \eqref{eq:Gammas}, we conclude
\begin{equation}\label{eq:sandwich_inters}
\Gamma^{n,0}_\eps(A)\cup\cdots\cup\Gamma^{n,b-1}_\eps(A)\ \subset\ \speps A
\ \subset\ \Gamma^{n,0}_{\eps+\delta}(A)\cap\cdots\cap \Gamma^{n,b-1}_{\eps+\delta}(A).
\end{equation}
In examples one observes that the bound \eqref{eq:sandwich_inters} on $\speps A$ is sharper than any of \eqref{eq:Gammas}. 

\begin{example}
\label{ex:boundarysets}
We look at the following 2-periodic bi-infinite matrix with bandwidth $d=2$:
\[
A=\left(\begin{array}{c|cc|cc|cc}
\ddots&\ddots&\ddots&&\\
\hline
\smash\ddots&0&9&4&&\\
\smash\ddots&9&0&2&0\\
\hline
&0&2&0&9&4\\
&&0&9&0&2&\smash\ddots\\
\hline
&&&0&2&0&\smash\ddots\\[-2mm]
&&&&\ddots&\ddots&\ddots
\end{array}\right)
=
\left(\begin{array}{cc|cc|cc|c}
\ddots&\ddots&\ddots&&&\\
\smash\ddots&0&9&4&&\\
\hline
\smash\ddots&9&0&2&0&&\\
&0&2&0&9&4\\
\hline
&&0&9&0&2&\smash\ddots\\
&&&0&2&0&\smash\ddots\\
\hline
&&&&\ddots&\ddots&\ddots
\end{array}\right)
\]
The block size was chosen to be $b=d=2$, leading to the two different possibilities of block positioning ($c=0$ and $c=1$) shown above. Figure \ref{fig:forExmp2.1} below shows a plot of $\Gamma^{n,0}_\eps(A)$ and, for comparison, of $\Gamma^{n,1}_\eps(A)$, as well as $\Gamma_\epsilon^{n,0}(A)\cap \Gamma_\epsilon^{n,1}(A)$ for $n=6$. 
\end{example}

\begin{figure}[h]
\begin{center}
\includegraphics{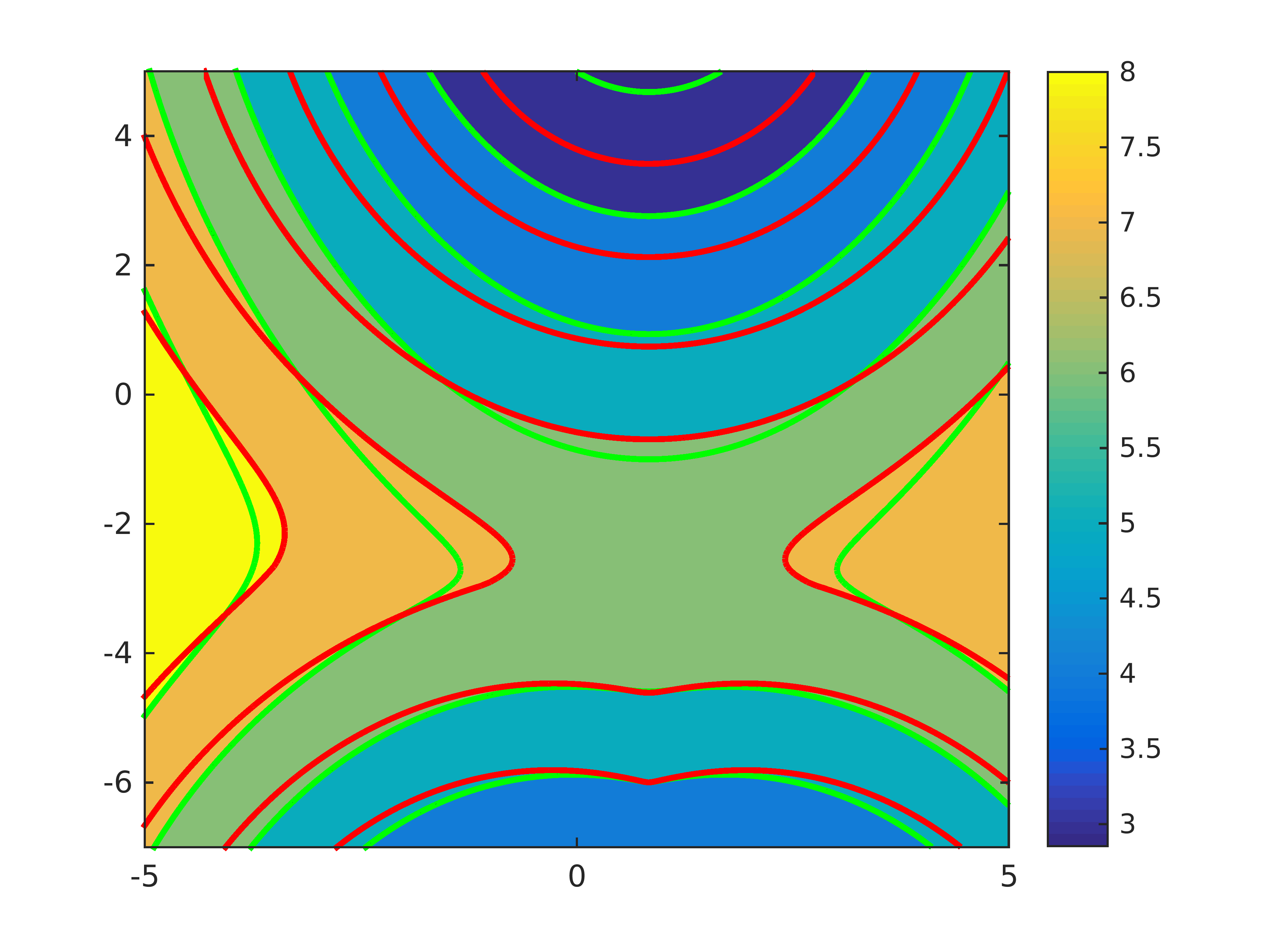}\\
\caption{Regarding Example \ref{ex:boundarysets}, we see the boundaries of the sets  $\Gamma^{n,0}_\eps(A)$ (dark/red line) and, for comparison, of $\Gamma^{n,1}_\eps(A)$ (light/green line), both for $n=6$ and $\eps=2,3,\ldots,8$. The colored areas denote $\Gamma_\epsilon^{n,0}(A)\cap \Gamma_\epsilon^{n,1}(A)$.}
\label{fig:forExmp2.1}
\end{center}
\end{figure}
This is why we suggest to look at all (instead of just one) of the inclusions \eqref{eq:Gammas}. Of course this improvement in quality of the bound on $\speps A$ increases the numerical costs by a factor of $b$. The next section shows how to compensate for that.


\section{The Algorithm} \label{sec:algorithm}
To simplify notation abbreviate, for $k\in\Z,n\in\N$ and $\lambda\in\C$,
\begin{align*}
 \begin{array}{cccc}A^k_\lambda := (A-\lambda I)|_{\ell^2(J^n_k)} :& \ell^2(J^n_k)&\rightarrow& \ell^2(J_{k-d}^{n+2d})\\&\cong&&\cong\\&\C^n&&\C^{n+2d}\end{array}
\end{align*}
and treat $A^k_\lambda$ as a finite rectangular matrix. We define $\overline{A}^k_\lambda:=(A-\lambda I)^*|_{\ell^2(J^n_k)}$ analogously.\\
As has been described in the previous section, we need to approximate $\nu(A^k_\lambda)$ and $\nu(\overline{A}^k_\lambda)$ for different values $\lambda\in\C$ and multiple consecutive values of $k$. This can be done by computing the smallest singular values $\sigma_n(A^k_\lambda)$ and $\sigma_n(\overline{A}^k_\lambda)$ which is strongly related to pseudospectra of rectangular matrices (\cite{TrefRectMatr}) and similar computational problems arise.\\
If the considered matrices $A^k_\lambda$ were square, we could compute the Schur decomposition of $A^k_0$ -- thus transforming $A^k_0$ into upper rectangular form -- while preserving the shift by $\lambda$. Afterwards we could compute $\sigma_n(A^k_\lambda)$ for multiple values of $\lambda\in\C$ using a bidiagonalization method \cite{nelabosnerPhd} on $(A_\lambda^k)^{-1}$. In the rectangular case though, no shift-preserving method to reduce $A_0^k$ to a simple form appears to be known and an inverse iteration is more difficult to implement for rectangular matrices.\\ 
We can however use the fact, that for each $\lambda\in\C$ we have two sequences $(A^k_\lambda)_k$ and $(\overline{A}^k_\lambda)_k$ each of which contains large overlaps between consecutive matrices. We will introduce an algorithm that takes advantage of this property.\\
We fix $\lambda\in\C$ and $n\in\N$ and abbreviate $A^{k}:=A^k_\lambda\in\C^{(n+2d)\times n}$ for $k\in\Z$.\\
Let $A^{k_0+1},A^{k_0+2},\ldots,A^{k_0+k_{\max}}$ be a finite sequence of matrices given by \eqref{eq:defGammaM}. W.l.o.g. we consider $k_0=0$. We can describe the overlapping property of these matrices by
 \begin{equation}
  A^{k}_{i,j}=A^{k+1}_{i-1,j-1},\; \text{for all } \left\{\begin{array}{l}1\leq k\leq k_{\max}-1\\ 2\leq i\leq n+2d=:m\\ 2\leq j\leq n.\end{array}\right. \label{eq:Ankshift}
 \end{equation}
Since $\nu(A^{k})=\sigma_n(A^{k})$, we are interested in computing the set
 \begin{align*}
 \{\sigma_n(A^{k})\}_{1\leq k\leq k_{\max}},
\end{align*}
where $\sigma_n$ denotes the smallest singular value, which can be approximated using a QR decomposition 
\begin{align}
Q^kA^{k}=R^{k}=\left(\begin{array}{c}{\tilde R}^k\\ \zero\end{array}\right),\text{ with }\tilde{R}^k\in\C^{n\times n}\text{ upper triangular}, Q^k\in\C^{m\times m}\text{ unitary}\label{eq:qrdec}
\end{align}
and applying an inverse Golub-Kahan-Lanczos-Bidiagonalization method (\cite{TemplAlgEigenv,MatrixComp83}), from now on abbreviated as GKLB method, to $\tilde R^k$ (i.e. applying the GKLB method to ${(\tilde R^k)}^{-1}$). Since this is a unitary transformation, the singular values of $A^k$ and $\tilde R^k$ are the same.
The inverse GKLB method requires solving two linear systems of equations in each iteration which can be achieved using backward-substitution, since $\tilde R^k$ is upper triangular.\\
Note that unlike convention we write $Q^kA^k=R^k$ instead of $(Q^k)^HA^k=R^k$ to simplify notation.
It is possible to compute a QR decomposition such that the banded structure of $A^k$ is preserved in $\tilde{R}^k$, i.e. $\tilde R^k$ has at most $2d+1$ consecutive non-zero diagonals. Therefore solving a linear system of equations involving $\tilde R^k$ requires only $\O(nd)$ flops. The QR decomposition \eqref{eq:qrdec} itself however requires $\O(nd^2)$ operations and is therefore the bottleneck of the algorithm for large $d$.\\
This bottleneck is addressed in the QH-shift-algorithm which we will develop in this section. The idea of the algorithm is to use Givens rotations to compute the factorization $Q^1A^{1}=H^1$, where $H^1\in\C^{m\times n}$ is an upper Hessenberg-matrix\footnote{We say a matrix $H\in\C^{m\times n}$ is an upper Hessenberg-matrix if $H_{i,j}=0$ for all $i>j+1$} with $2d+1$ consecutive non-zero diagonals, and then reuse these rotations to factorize $A^{2},A^{3},\ldots$ the same way.\\
Having factorized $A^{1},A^{2},\ldots$ into Hessenberg form using unitary transformations, we only need to apply $n$ additional Givens rotations to each matrix to arrive at the QR decomposition \eqref{eq:qrdec}. The total effort for each QR decomposition of $A^{2}, A^{3},\ldots$ is only $\O(nd)$ instead of $\O(nd^2)$.

{\bf Preliminaries.}
We will only use Givens rotations acting on consecutive rows and define a rotation on the $i$th and $(i+1)$st row by the mapping
\begin{align*}
\begin{array}{llll}
G_{i}:&\overline\D\times\overline\D&\rightarrow&\C^{m\times m}\\
&(c,s)&\mapsto&G_{i}(c,s).
\end{array}
\end{align*}
and
\begin{align*}
G_i(c,s)=\bordermatrix{
&   &  &i       &i+1    &&    \cr
&  1   & \cdots &    0   &     0   & \cdots &    0   \cr
&\vdots & \ddots & \vdots &         \vdots &        & \vdots \cr
i&0   & \cdots &    c   &     s   & \cdots &    0   \cr
i+1&0   & \cdots &   -\overline{s}    &    \overline{c}   & \cdots &    0   \cr
&\vdots &        & \vdots &         \vdots & \ddots & \vdots \cr
&0   & \cdots &    0   &     0   & \cdots &    1   
}
\end{align*}
where $\overline{\D}:=\{z\in\C:\,\abs{z}\leq 1\}$ is the closed complex unit disc. Details on the choice of $c,s$ can be found in standard literature \cite{MatrixComp83, MatrixComp08}.
To simplify notation we will, in most cases, write $G_i\equiv G_i(c,s)$, if the choice of $c,s$ is clear from the context. This naturally leads to the problem of possibly having multiple rotations on the same row, each having different entries $c,s$ and we hope that it will be clear from the context that these Givens rotations are not the same.\\
In the interest of readability we will mainly use the arrow-notation introduced by Raf Vandebril et al. in \cite{MatrixComp08,RotFigPackage}:
\begin{align}
 \begin{array}{c@{\hspace{1mm}}|c@{\hspace{1mm}}c@{\hspace{1mm}}c@{\hspace{1mm}}c@{\hspace{1mm}}}
1&&&&\Rc      \\[-0.05cm]
2&&&\Rc&\rc   \\[-0.05cm]  
3&&\Rc&\rc&      \\[-0.05cm]
4&\Rc&\rc&& \\[-0.05cm]
5&\rc&&& \\[-0.05cm]\cline{1-5}
&4&3&2&1
\end{array}\label{eq:arrowIntr}
\end{align}
The arrows in \eqref{eq:arrowIntr} each depict a Givens rotation operation, acting on the two rows in which the arrow is drawn (see axis of ordinates). The order of application of these rotations is described in the abscissa, i.e. from right to left, so that \eqref{eq:arrowIntr} represents the product $G_4G_3G_2G_1$.
It is important to note the order of application of the Givens rotations, since they do not commute in general unless they act on disjoint couples of rows:
\begin{align}
\begin{array}{c@{\hspace{1mm}}c@{\hspace{1mm}}}
\Rc&      \\[-0.05cm]
\rc&   \\[-0.05cm]  
&\Rc      \\[-0.05cm]
&\rc
\end{array}
\left[\begin{array}{c@{\hspace{1mm}}c@{\hspace{1mm}}}
\times & \times \\[-0.05cm]
\times & \times\\[-0.05cm]
\times & \times\\[-0.05cm]
\times & \times
\end{array}\right]
=
\begin{array}{c@{\hspace{1mm}}c@{\hspace{1mm}}}
&\Rc      \\[-0.05cm]
&\rc   \\[-0.05cm]  
\Rc&      \\[-0.05cm]
\rc&
\end{array}
\left[\begin{array}{c@{\hspace{1mm}}c@{\hspace{1mm}}}
\times & \times \\[-0.05cm]
\times & \times\\[-0.05cm]
\times & \times\\[-0.05cm]
\times & \times
\end{array}\right],\text{ but }
\begin{array}{c@{\hspace{1mm}}c@{\hspace{1mm}}}
\Rc&      \\[-0.05cm]
\rc&\Rc   \\[-0.05cm]  
&\rc      \\[-0.05cm]
&
\end{array}
\left[\begin{array}{c@{\hspace{1mm}}c@{\hspace{1mm}}}
\times & \times \\[-0.05cm]
\times & \times\\[-0.05cm]
\times & \times\\[-0.05cm]
\times & \times
\end{array}\right]
\neq
\begin{array}{c@{\hspace{1mm}}c@{\hspace{1mm}}}
&\Rc      \\[-0.05cm]
\Rc&\rc   \\[-0.05cm]  
\rc&      \\[-0.05cm]
&
\end{array}
\left[\begin{array}{c@{\hspace{1mm}}c@{\hspace{1mm}}}
\times & \times \\[-0.05cm]
\times & \times\\[-0.05cm]
\times & \times\\[-0.05cm]
\times & \times
\end{array}\right]\label{eq:givenscommute}
\end{align}
We say that a product $G_{i_1},G_{i_2},\ldots,G_{i_l}$ is a {\sl descending}, respectively {\sl ascending}, {\sl sequence of Givens rotations} of {\sl length} $l$, if $i_{p+1}=i_p-1$, respectively $i_{p+1}=i_p+1$, for $p=1,\ldots,l-1$. \eqref{eq:arrowIntr} is an example of a descending sequence of length $4$.
\pagebreak  

\begin{example}
The following Givens rotations can be written as a product of $3$ descending sequences of length $4$ or as a product of $4$ ascending sequences of length $3$.
\begin{align*}
(G_4 G_3 G_2 G_1)(G_5 G_4 G_3 G_2)(G_6 G_5 G_4 G_3)=&
\begin{array}{c@{\hspace{1mm}}c@{\hspace{1mm}}c@{\hspace{1mm}}c@{\hspace{1mm}};{2pt/2pt}c@{\hspace{1mm}}c@{\hspace{1mm}}c@{\hspace{1mm}}c@{\hspace{1mm}};{2pt/2pt}c@{\hspace{1mm}}c@{\hspace{1mm}}
  c@{\hspace{1mm}}c@{\hspace{1mm}}}
   &   &   &\Rc&   &   &   &   &   &   &   &   \\[-0.05cm]
   &   &\Rc&\rc&   &   &   &\Rc&   &   &   &   \\[-0.05cm] 
   &\Rc&\rc&   &   &   &\Rc&\rc&   &   &   &\Rc\\[-0.05cm]
\Rc&\rc&   &   &   &\Rc&\rc&   &   &   &\Rc&\rc\\[-0.05cm]
\rc&   &   &   &\Rc&\rc&   &   &   &\Rc&\rc&   \\[-0.05cm]
   &   &   &   &\rc&   &   &   &\Rc&\rc&   &   \\[-0.05cm]
   &   &   &   &   &   &   &   &\rc&   &   &   
\end{array}=
\underbrace{\begin{array}{c@{\hspace{1mm}}c@{\hspace{1mm}}c@{\hspace{1mm}}c@{\hspace{1mm}}c@{\hspace{1mm}}c@{\hspace{1mm}}
  c@{\hspace{1mm}}c@{\hspace{1mm}}}
   &   &   &   &   &\Rc&   &   \\[-0.05cm]
   &   &   &   &\Rc&\rc&\Rc&   \\[-0.05cm] 
   &   &   &\Rc&\rc&\Rc&\rc&\Rc\\[-0.05cm]
   &   &\Rc&\rc&\Rc&\rc&\Rc&\rc\\[-0.05cm]
   &   &\rc&\Rc&\rc&\Rc&\rc&   \\[-0.05cm]
   &   &   &\rc&\Rc&\rc&   &   \\[-0.05cm]
   &   &   &   &\rc&   &   &   
\end{array}}_{(*)}\\
=\begin{array}{c@{\hspace{1mm}}c@{\hspace{1mm}}c@{\hspace{1mm}};{2pt/2pt}c@{\hspace{1mm}}c@{\hspace{1mm}}c@{\hspace{1mm}};{2pt/2pt}c@{\hspace{1mm}}c@{\hspace{1mm}}c@{\hspace{1mm}};{2pt/2pt}c@{\hspace{1mm}}
  c@{\hspace{1mm}}c@{\hspace{1mm}}}
   &   &   &   &   &   &   &   &   &\Rc&   &   \\[-0.05cm]
   &   &   &   &   &   &\Rc&   &   &\rc&\Rc&   \\[-0.05cm] 
   &   &   &\Rc&   &   &\rc&\Rc&   &   &\rc&\Rc\\[-0.05cm]
\Rc&   &   &\rc&\Rc&   &   &\rc&\Rc&   &   &\rc\\[-0.05cm]
\rc&\Rc&   &   &\rc&\Rc&   &   &\rc&   &   &   \\[-0.05cm]
   &\rc&\Rc&   &   &\rc&   &   &   &   &   &   \\[-0.05cm]
   &   &\rc&   &   &   &   &   &   &   &   &   
\end{array}=&
(G_4 G_5 G_6)(G_3 G_4 G_5)(G_2 G_3 G_4)(G_1 G_2 G_3)
\end{align*}
Note that the rotations can be written in this compact form since the order of rotations that are in the same column of $(*)$ is irrelevant by \eqref{eq:givenscommute}.
\end{example}\\
We illustrate the algorithm using example matrices with parameters $n=7$, $d=2$, i.e. matrices from $\C^{11\times 7}$, which is just large enough to visualize the procedure. Most transformations which are applied in this algorithm are easy to see but technical to prove, and most proofs have therefore been omitted.\\
The algorithm is divided into several steps, each representing one matrix from the sequence $\{A^{k}\}_{1\leq k\leq k_{\max}}$.

\textbf{Step 1:}
We start the first step by computing a QH factorization of $A^{1}$ using consecutive Givens rotations. The number of subdiagonals\footnote{Where we define the main diagonal as the set $\{A^k_{i,i}:\,1\leq i\leq n\}$} is $2d$, we therefore require $2d-1$ sequences of Givens rotations to achieve Hessenberg form:
\begin{align}
\begin{array}{c@{\hspace{1mm}}c@{\hspace{1mm}}c@{\hspace{1mm}}
  c@{\hspace{1mm}}c@{\hspace{1mm}}c@{\hspace{1mm}}c@{\hspace{1mm}}
  c@{\hspace{1mm}}c@{\hspace{1mm}}c@{\hspace{1mm}}c@{\hspace{1mm}}}
&   &   &   &   &   &   &   &   &   &   \\[-0.05cm] 
&   &   &   &   &   &   &   &   &   &   \\[-0.05cm]
&   &   &   &   &   &   &   &   &   &   \\[-0.05cm]
&   &   &   &   &   &   &   &   &   &\Rc\\[-0.05cm]
&   &   &   &   &   &   &   &   &\Rc&\rc\\[-0.05cm] 
&   &   &   &   &   &   &   &\Rc&\rc&   \\[-0.05cm]
&   &   &   &   &   &   &\Rc&\rc&   &   \\[-0.05cm] 
&   &   &   &   &   &\Rc&\rc&   &   &   \\[-0.05cm]
&   &   &   &   &\Rc&\rc&   &   &   &   \\[-0.05cm] 
&   &   &   &\Rc&\rc&   &   &   &   &   \\[-0.05cm] 
&   &   &   &\rc&   &   &   &   &   &        
\end{array}
\left[\begin{array}{c@{\hspace{1mm}}c@{\hspace{1mm}}
                  c@{\hspace{1mm}}c@{\hspace{1mm}}c@{\hspace{1mm}}c@{\hspace{1mm}}c@{\hspace{1mm}}}
\times & 0& 0& 0&0& 0& 0 \\[-0.05cm]
\times & \times& 0& 0&0& 0& 0 \\[-0.05cm]
\times & \times& \times& 0&0& 0& 0 \\[-0.05cm]
\times & \times& \times& \times&0& 0& 0 \\[-0.05cm]
\times & \times& \times& \times&\times& 0& 0 \\[-0.05cm]
0&\times & \times& \times& \times&\times& 0 \\[-0.05cm]
0&0&\times & \times& \times& \times&\times\\[-0.05cm]
0&0&0 & \times& \times& \times&\times\\[-0.05cm]
0&0&0& 0& \times& \times&\times\\[-0.05cm]
0&0&0 & 0&0& \times&\times\\[-0.05cm]
0&0&0 & 0&0&0&\times
\end{array}\right]
\quad=&
\quad
\left[\begin{array}{c@{\hspace{1mm}}c@{\hspace{1mm}}
                  c@{\hspace{1mm}}c@{\hspace{1mm}}c@{\hspace{1mm}}c@{\hspace{1mm}}c@{\hspace{1mm}}}
\times & 0& 0& 0&0& 0& 0 \\[-0.05cm]
\times & \times& 0& 0&0& 0& 0 \\[-0.05cm]
\times & \times& \times& 0&0& 0& 0 \\[-0.05cm]
\times & \times& \times& \times&\times& 0& 0 \\[-0.05cm]
0 & \times& \times& \times&\times& \times& 0 \\[-0.05cm]
0&0 & \times& \times& \times&\times& \times \\[-0.05cm]
0&0&0 & \times& \times& \times&\times\\[-0.05cm]
0&0&0 & 0& \times& \times&\times\\[-0.05cm]
0&0&0& 0& 0& \times&\times\\[-0.05cm]
0&0&0 & 0&0& 0&\times\\[-0.05cm]
0&0&0 & 0&0&0&0
\end{array}\right]\nonumber\\
\Downarrow&\nonumber\\
 Q^1 A^{1} =
\begin{array}{c@{\hspace{1mm}}c@{\hspace{1mm}}c@{\hspace{1mm}}
  c@{\hspace{1mm}}c@{\hspace{1mm}}c@{\hspace{1mm}}c@{\hspace{1mm}}
  c@{\hspace{1mm}}c@{\hspace{1mm}}c@{\hspace{1mm}}c@{\hspace{1mm}}}
&   &   &   &   &   &   &   &   &   &   \\[-0.05cm] 
&   &   &   &   &   &   &   &\Rc&   &   \\[-0.05cm]
&   &   &   &   &   &   &\Rc&\rc&\Rc&   \\[-0.05cm]
&   &   &   &   &   &\Rc&\rc&\Rc&\rc&\Rc\\[-0.05cm]
&   &   &   &   &\Rc&\rc&\Rc&\rc&\Rc&\rc\\[-0.05cm] 
&   &   &   &\Rc&\rc&\Rc&\rc&\Rc&\rc&   \\[-0.05cm]
&   &   &\Rc&\rc&\Rc&\rc&\Rc&\rc&   &   \\[-0.05cm] 
&   &\Rc&\rc&\Rc&\rc&\Rc&\rc&   &   &   \\[-0.05cm]
&   &\rc&\Rc&\rc&\Rc&\rc&   &   &   &   \\[-0.05cm] 
&   &   &\rc&\Rc&\rc&   &   &   &   &   \\[-0.05cm] 
&   &   &   &\rc&   &   &   &   &   &        
\end{array}
\left[\begin{array}{c@{\hspace{1mm}}c@{\hspace{1mm}}
                  c@{\hspace{1mm}}c@{\hspace{1mm}}c@{\hspace{1mm}}c@{\hspace{1mm}}c@{\hspace{1mm}}}
\times & 0& 0& 0&0& 0& 0 \\[-0.05cm]
\times & \times& 0& 0&0& 0& 0 \\[-0.05cm]
\times & \times& \times& 0&0& 0& 0 \\[-0.05cm]
\times & \times& \times& \times&0& 0& 0 \\[-0.05cm]
\times & \times& \times& \times&\times& 0& 0 \\[-0.05cm]
0&\times & \times& \times& \times&\times& 0 \\[-0.05cm]
0&0&\times & \times& \times& \times&\times\\[-0.05cm]
0&0&0 & \times& \times& \times&\times\\[-0.05cm]
0&0&0& 0& \times& \times&\times\\[-0.05cm]
0&0&0 & 0&0& \times&\times\\[-0.05cm]
0&0&0 & 0&0&0&\times
\end{array}\right]
\quad=&
\quad
\left[
\begin{array}{c@{\hspace{1mm}}c@{\hspace{1mm}}c@{\hspace{1mm}}c@{\hspace{1mm}}c@{\hspace{1mm}}c@{\hspace{1mm}}c@{\hspace{1mm}}}
\times & 0      & 0      & 0      & 0     & 0     & 0      \\[-0.05cm]
\times & \times & \times & \times & \times& 0     & 0      \\[-0.05cm]
0      & \times & \times & \times & \times& \times& 0      \\[-0.05cm]
0      & 0      & \times & \times & \times& \times& \times \\[-0.05cm]
0      & 0      & 0      & \times & \times& \times& \times \\[-0.05cm]
0      & 0      & 0      & 0      & \times& \times& \times \\[-0.05cm]
0      & 0      & 0      & 0      & 0     & \times& \times \\[-0.05cm]
0      & 0      & 0      & 0      & 0     & 0     & \times \\[-0.05cm]
0      & 0      & 0      & 0      & 0     & 0     & 0      \\[-0.05cm]
0      & 0      & 0      & 0      & 0     & 0     & 0      \\[-0.05cm]
0      & 0      & 0      & 0      & 0     & 0     & 0      
\end{array}\right]=H^1.\label{eq:qhshiftstep1}
\end{align}
Notice that, since the right hand side is of Hessenberg form, no rotations acting on the first row are required and the first row of $Q^1$ is the unit vector $e_1^T$. Since $Q^1$ is unitary it has the form
\begin{align}
Q^1=\left(\begin{array}{cc}1&\zero\\\zero&\tilde{Q}^1\end{array}\right),\;\;\tilde{Q}^1\in \C^{m-1\times m-1}\label{eq:q1structure}
\end{align}
The computational effort of this step consists of the computation and application of $n(2d-1)$ Givens rotations. Because of the band structure the number of flops required is $\O(nd^2)$.\\
We can now easily compute a QR decomposition by applying $n$ additional Givens rotations to the matrix $H^1$. This is done in $\O(nd)$ flops.

\textbf{Step 2}:
Since $A^1$ and $A^2$ overlap in all but one row and column each, we can derive $A^2$ from $A^1$ by cutting off the first row and column, shifting all values by one entry to the top left (as in \eqref{eq:Ankshift}) and add a new row and column at the end. More precisely, let
\begin{align*}
 C_p:=\left(\begin{array}{cc}\zero&I_{p-1}\\1&\zero\end{array}\right): \left(\begin{array}{c} x_1\\\vdots\\x_{p-1}\\x_p\end{array}\right)\mapsto \left(\begin{array}{c} x_2\\\vdots\\x_{p}\\x_1\end{array}\right)
\end{align*}
denote the circulant backward shift of size $p$. Then $\hat{A}^2:=C_mA^1C_n^{-1}$ differs from $A^2$ only in the last column (the first $n-1$ entries in the last row are zero in both matrices) and satisfies \eqref{eq:Ankshift}. We illustrate this step $A^1\rightarrow \hat{A}^2\rightarrow A^2$ as follows, where $-$ and $+$ denote the entries lost and gained respectively:
\begin{align*}
\underbrace{
\left[\begin{array}{c@{\hspace{1mm}}c@{\hspace{1mm}}
                  c@{\hspace{1mm}}c@{\hspace{1mm}}c@{\hspace{1mm}}c@{\hspace{1mm}}c@{\hspace{1mm}}}
- & 0& 0& 0&0& 0& 0 \\[-0.05cm]
- & \times& 0& 0&0& 0& 0 \\[-0.05cm]
- & \times& \times& 0&0& 0& 0 \\[-0.05cm]
- & \times& \times& \times&0& 0& 0 \\[-0.05cm]
- & \times& \times& \times&\times& 0& 0 \\[-0.05cm]
0&\times & \times& \times& \times&\times& 0 \\[-0.05cm]
0&0&\times & \times& \times& \times&\times\\[-0.05cm]
0&0&0 & \times& \times& \times&\times\\[-0.05cm]
0&0&0& 0& \times& \times&\times\\[-0.05cm]
0&0&0 & 0&0& \times&\times\\[-0.05cm]
0&0&0 & 0&0&0&\times
\end{array}\right]}_{=A^1}\rightarrow
\underbrace{
\left[\begin{array}{c@{\hspace{1mm}}c@{\hspace{1mm}}
                  c@{\hspace{1mm}}c@{\hspace{1mm}}c@{\hspace{1mm}}c@{\hspace{1mm}}c@{\hspace{1mm}}}
\times & 0& 0& 0&0& 0& - \\[-0.05cm]
\times & \times& 0& 0&0& 0& - \\[-0.05cm]
\times & \times& \times& 0&0& 0& - \\[-0.05cm]
\times & \times& \times& \times&0& 0& - \\[-0.05cm]
\times & \times& \times& \times&\times& 0& 0 \\[-0.05cm]
0&\times & \times& \times& \times&\times& 0 \\[-0.05cm]
0&0&\times & \times& \times& \times&0\\[-0.05cm]
0&0&0 & \times& \times& \times&0\\[-0.05cm]
0&0&0& 0& \times& \times&0\\[-0.05cm]
0&0&0 & 0&0& \times&0\\[-0.05cm]
0&0&0 & 0&0&0&-
\end{array}\right]}_{=\hat{A}^2}\rightarrow
\underbrace{\left[\begin{array}{c@{\hspace{1mm}}c@{\hspace{1mm}}
                  c@{\hspace{1mm}}c@{\hspace{1mm}}c@{\hspace{1mm}}c@{\hspace{1mm}}c@{\hspace{1mm}}}
\times & 0& 0& 0&0& 0& 0 \\[-0.05cm]
\times & \times& 0& 0&0& 0& 0 \\[-0.05cm]
\times & \times& \times& 0&0& 0& 0 \\[-0.05cm]
\times & \times& \times& \times&0& 0& 0 \\[-0.05cm]
\times & \times& \times& \times&\times& 0& 0 \\[-0.05cm]
0&\times & \times& \times& \times&\times& 0 \\[-0.05cm]
0&0&\times & \times& \times& \times&+\\[-0.05cm]
0&0&0 & \times& \times& \times&+\\[-0.05cm]
0&0&0& 0& \times& \times&+\\[-0.05cm]
0&0&0 & 0&0& \times&+\\[-0.05cm]
0&0&0 & 0&0&0&+
\end{array}\right]}_{=A^2}
\end{align*}
We apply the same transformation to the factorization $Q^1A^1=H^1$:
\begin{align}
 Q^1A^1=H^1\Rightarrow \underbrace{C_m Q^1C_m^{-1}}_{=:\hat{Q}^2}\underbrace{C_mA^1C_n^{-1}}_{=\hat{A}^2}=\underbrace{C_mH^1C_n^{-1}}_{=:\hat{H}^2}.\label{eq:qhshiftfact}
\end{align}
Notice that $\hat{Q}^2$ is again unitary and can be written as
\begin{align*}
 \hat{Q}^2=C_mQ^1C_m^{-1}=C_m\left(\begin{array}{cc} 1&\zero\\\zero&\tilde{Q}^1\end{array}\right)C_m^{-1}=\left(\begin{array}{cc} \tilde{Q}^1&\zero\\\zero&1\end{array}\right).
\end{align*}
The matrix $\hat{Q}^2$ consists of the same sequences of Givens rotations as before, where all Givens rotations have been shifted up by one row.
We write the factorization \eqref{eq:qhshiftfact} as
\begin{align}
 \hat{Q}^2\hat{A}^2=\left(\begin{array}{c|c} \tilde{Q}^1&\zero\\\hline\zero&1\end{array}\right)\cdot \left(\begin{array}{clc|c}|&&|&|\\ \hat{a}^2_1&\cdots&\hat{a}^2_{n-1}&\hat{a}^2_n\\|&&|&|\\\hline0&\cdots&0&\hat{a}^2_{n,n}\end{array}\right)=
\left(\begin{array}{clc|c}|&&|&|\\ \tilde{Q}^1\hat{a}^2_1&\cdots&\tilde{Q}^1\hat{a}^2_{n-1}&\tilde{Q}^1\hat{a}^2_n\\|&&|&|\\\hline0&\cdots&0&\hat{a}^2_{n,n}\end{array}\right)=\hat{H}^2,\label{eq:qhshiftfactbig}
\end{align}
where $\hat{a}^2_i$ denotes the $i$th column of $\hat{A}^2$ without the last row. Notice that, by \eqref{eq:qhshiftfact}, $\hat{H}^2$ is again of upper Hessenberg form everywhere except in the last column. We will now replace $\hat{A}^2$ with $A^2$ in \eqref{eq:qhshiftfact} and \eqref{eq:qhshiftfactbig} which leads to $\hat{Q}^2A^2=:\tilde{H}^2$. As can be seen in \eqref{eq:qhshiftfactbig}, the matrices $\hat{H}^2$ and $\tilde{H}^2$ only differ in the last column because $\hat{A}^2$ and $A^2$ only differ in the last column. These new values have to be computed by applying $\hat{Q}^2$ to the last column of $A^2$. These are the only values which have to be calculated in this transformation and there is a fill-in of at most $2d-1$ non-zero values.\\
We illustrate this entire procedure as follows, where $+$ denotes the fill-in produced by applying $\hat{Q}^2$ to the last column of $A^2$:
\begin{align}
Q^1A^1=\begin{array}{c@{\hspace{1mm}}c@{\hspace{1mm}}c@{\hspace{1mm}}
  c@{\hspace{1mm}}c@{\hspace{1mm}}c@{\hspace{1mm}}c@{\hspace{1mm}}
  c@{\hspace{1mm}}c@{\hspace{1mm}}c@{\hspace{1mm}}c@{\hspace{1mm}}}
&   &   &   &   &   &   &   &   &   &   \\[-0.05cm] 
&   &   &   &   &   &   &   &\Rc&   &   \\[-0.05cm]
&   &   &   &   &   &   &\Rc&\rc&\Rc&   \\[-0.05cm]
&   &   &   &   &   &\Rc&\rc&\Rc&\rc&\Rc\\[-0.05cm]
&   &   &   &   &\Rc&\rc&\Rc&\rc&\Rc&\rc\\[-0.05cm] 
&   &   &   &\Rc&\rc&\Rc&\rc&\Rc&\rc&   \\[-0.05cm]
&   &   &\Rc&\rc&\Rc&\rc&\Rc&\rc&   &   \\[-0.05cm] 
&   &\Rc&\rc&\Rc&\rc&\Rc&\rc&   &   &   \\[-0.05cm]
&   &\rc&\Rc&\rc&\Rc&\rc&   &   &   &   \\[-0.05cm] 
&   &   &\rc&\Rc&\rc&   &   &   &   &   \\[-0.05cm] 
&   &   &   &\rc&   &   &   &   &   &        
\end{array}
\left[\begin{array}{c@{\hspace{1mm}}c@{\hspace{1mm}}
                  c@{\hspace{1mm}}c@{\hspace{1mm}}c@{\hspace{1mm}}c@{\hspace{1mm}}c@{\hspace{1mm}}}
\times & 0& 0& 0&0& 0& 0 \\[-0.05cm]
\times & \times& 0& 0&0& 0& 0 \\[-0.05cm]
\times & \times& \times& 0&0& 0& 0 \\[-0.05cm]
\times & \times& \times& \times&0& 0& 0 \\[-0.05cm]
\times & \times& \times& \times&\times& 0& 0 \\[-0.05cm]
0&\times & \times& \times& \times&\times& 0 \\[-0.05cm]
0&0&\times & \times& \times& \times&\times\\[-0.05cm]
0&0&0 & \times& \times& \times&\times\\[-0.05cm]
0&0&0& 0& \times& \times&\times\\[-0.05cm]
0&0&0 & 0&0& \times&\times\\[-0.05cm]
0&0&0 & 0&0&0&\times
\end{array}\right]
\quad=&
\quad
\left[
\begin{array}{c@{\hspace{1mm}}c@{\hspace{1mm}}c@{\hspace{1mm}}c@{\hspace{1mm}}c@{\hspace{1mm}}c@{\hspace{1mm}}c@{\hspace{1mm}}}
\times & 0      & 0      & 0      & 0     & 0     & 0      \\[-0.05cm]
\times & \times & \times & \times & \times& 0     & 0      \\[-0.05cm]
0      & \times & \times & \times & \times& \times& 0      \\[-0.05cm]
0      & 0      & \times & \times & \times& \times& \times \\[-0.05cm]
0      & 0      & 0      & \times & \times& \times& \times \\[-0.05cm]
0      & 0      & 0      & 0      & \times& \times& \times \\[-0.05cm]
0      & 0      & 0      & 0      & 0     & \times& \times \\[-0.05cm]
0      & 0      & 0      & 0      & 0     & 0     & \times \\[-0.05cm]
0      & 0      & 0      & 0      & 0     & 0     & 0      \\[-0.05cm]
0      & 0      & 0      & 0      & 0     & 0     & 0      \\[-0.05cm]
0      & 0      & 0      & 0      & 0     & 0     & 0      
\end{array}\right]=H^1\nonumber\\
\Downarrow&\nonumber\\
\hat{Q}^2A^2=\begin{array}{c@{\hspace{1mm}}c@{\hspace{1mm}}c@{\hspace{1mm}}
  c@{\hspace{1mm}}c@{\hspace{1mm}}c@{\hspace{1mm}}c@{\hspace{1mm}}
  c@{\hspace{1mm}}c@{\hspace{1mm}}c@{\hspace{1mm}}c@{\hspace{1mm}}}
&   &   &   &   &   &   &   &\Rr&   &   \\[-0.05cm]
&   &   &   &   &   &   &\Rr&\rr&\Rc&   \\[-0.05cm]
&   &   &   &   &   &\Rr&\rr&\Rc&\rc&\Rc\\[-0.05cm]
&   &   &   &   &\Rr&\rr&\Rc&\rc&\Rc&\rc\\[-0.05cm] 
&   &   &   &\Rr&\rr&\Rc&\rc&\Rc&\rc&   \\[-0.05cm]
&   &   &\Rr&\rr&\Rc&\rc&\Rc&\rc&   &   \\[-0.05cm] 
&   &\Rr&\rr&\Rc&\rc&\Rc&\rc&   &   &   \\[-0.05cm]
&   &\rr&\Rc&\rc&\Rc&\rc&   &   &   &   \\[-0.05cm] 
&   &   &\rc&\Rc&\rc&   &   &   &   &   \\[-0.05cm] 
&   &   &   &\rc&   &   &   &   &   &   \\[-0.05cm] 
&   &   &   &   &   &   &   &   &   & 
\end{array}
\left[\begin{array}{c@{\hspace{1mm}}c@{\hspace{1mm}}
                  c@{\hspace{1mm}}c@{\hspace{1mm}}c@{\hspace{1mm}}c@{\hspace{1mm}}c@{\hspace{1mm}}}
\times & 0& 0& 0&0& 0& 0 \\[-0.05cm]
\times & \times& 0& 0&0& 0& 0 \\[-0.05cm]
\times & \times& \times& 0&0& 0& 0 \\[-0.05cm]
\times & \times& \times& \times&0& 0& 0 \\[-0.05cm]
\times & \times& \times& \times&\times& 0& 0 \\[-0.05cm]
0&\times & \times& \times& \times&\times& 0 \\[-0.05cm]
0&0&\times & \times& \times& \times&\times\\[-0.05cm]
0&0&0 & \times& \times& \times&\times\\[-0.05cm]
0&0&0& 0& \times& \times&\times\\[-0.05cm]
0&0&0 & 0&0& \times&\times\\[-0.05cm]
0&0&0 & 0&0&0&\times
\end{array}\right]
\quad=&
\quad
\left[
\begin{array}{c@{\hspace{1mm}}c@{\hspace{1mm}}c@{\hspace{1mm}}c@{\hspace{1mm}}c@{\hspace{1mm}}c@{\hspace{1mm}}c@{\hspace{1mm}}}
\times & \times & \times & \times & 0& 0     & 0      \\[-0.05cm]
\times & \times & \times & \times & \times& 0& 0      \\[-0.05cm]
0      & \times & \times & \times & \times& \times& 0\\[-0.05cm]
0      & 0      & \times & \times & \times& \times& + \\[-0.05cm]
0      & 0      & 0      & \times & \times& \times& + \\[-0.05cm]
0      & 0      & 0      & 0      & \times& \times& + \\[-0.05cm]
0      & 0      & 0      & 0      & 0     & \times& \times \\[-0.05cm]
0      & 0      & 0      & 0      & 0     & 0     & \times \\[-0.05cm]
0      & 0      & 0      & 0      & 0     & 0     & \times \\[-0.05cm]
0      & 0      & 0      & 0      & 0     & 0     & \times \\[-0.05cm]
0      & 0      & 0      & 0      & 0     & 0     & \times 
\end{array}\right]=\tilde{H}^2\label{eq:qhshiftstep2}
\end{align} 
We anticipate that in the next step we would like to apply the same shift again. However, since $\hat{Q}^2$ acts on the first row, the requirement \eqref{eq:q1structure} does not hold. If we were to naively shift all values of $\hat{Q}^2$ again to the top left by one entry and add $e_m^T$ in the last row and column, the resulting matrix $\hat{Q}^3$ would not be unitary. Figuratively speaking we would cut one Givens-rotation in half, since there can be no rotation acting on the ``zero''th row.\footnote{It is of course possible to allow Givens rotations acting on non-consecutive rows. However these rotations are difficult to remove leading to an ever increasing number of rotations.}\\
Therefore we have to remove the Givens rotation acting on the first row in $\hat{Q}^2$ in the left-most descending sequence, which is marked as gray in \eqref{eq:qhshiftstep2}. This can be done by applying the inverse rotations. The rotations in this sequence do not commute, since they are ordered consecutively. Thus we remove the entire sequence and add it again only this time starting in the second and ending in the $(n-1)$st row. This again costs $\O(nd)$.\\
Starting with \eqref{eq:qhshiftstep2} we remove the left-most descending sequence, which results in a fill-in in the 2nd subdiagonal ($+$ signs) and can be illustrated as
\begin{align*}
\begin{array}{c@{\hspace{1mm}}c@{\hspace{1mm}}c@{\hspace{1mm}}c@{\hspace{1mm}}
  c@{\hspace{1mm}}c@{\hspace{1mm}}c@{\hspace{1mm}}c@{\hspace{1mm}}
  c@{\hspace{1mm}}c@{\hspace{1mm}}c@{\hspace{1mm}}c@{\hspace{1mm}}}
   &   &   &   &   &   &   &   &   &   &   &   \\[-0.05cm]
   &   &   &   &   &   &   &   &   &   &\Rc&   \\[-0.05cm]
   &   &   &   &   &   &   &   &   &\Rc&\rc&\Rc\\[-0.05cm]
   &   &   &   &   &   &   &   &\Rc&\rc&\Rc&\rc\\[-0.05cm] 
   &   &   &   &   &   &   &\Rc&\rc&\Rc&\rc&   \\[-0.05cm]
   &   &   &   &   &   &\Rc&\rc&\Rc&\rc&   &   \\[-0.05cm] 
   &   &   &   &   &\Rc&\rc&\Rc&\rc&   &   &   \\[-0.05cm]
   &   &   &   &\Rc&\rc&\Rc&\rc&   &   &   &   \\[-0.05cm] 
   &   &   &   &\rc&\Rc&\rc&   &   &   &   &   \\[-0.05cm] 
   &   &   &   &   &\rc&   &   &   &   &   &   \\[-0.05cm] 
   &   &   &   &   &   &   &   &   &   &   & 
\end{array}
\left[\begin{array}{c@{\hspace{1mm}}c@{\hspace{1mm}}
                  c@{\hspace{1mm}}c@{\hspace{1mm}}c@{\hspace{1mm}}c@{\hspace{1mm}}c@{\hspace{1mm}}}
\times & 0& 0& 0&0& 0& 0 \\[-0.05cm]
\times & \times& 0& 0&0& 0& 0 \\[-0.05cm]
\times & \times& \times& 0&0& 0& 0 \\[-0.05cm]
\times & \times& \times& \times&0& 0& 0 \\[-0.05cm]
\times & \times& \times& \times&\times& 0& 0 \\[-0.05cm]
0&\times & \times& \times& \times&\times& 0 \\[-0.05cm]
0&0&\times & \times& \times& \times&\times\\[-0.05cm]
0&0&0 & \times& \times& \times&\times\\[-0.05cm]
0&0&0& 0& \times& \times&\times\\[-0.05cm]
0&0&0 & 0&0& \times&\times\\[-0.05cm]
0&0&0 & 0&0&0&\times
\end{array}\right]
\quad=&
\quad
\left[
\begin{array}{c@{\hspace{1mm}}c@{\hspace{1mm}}c@{\hspace{1mm}}c@{\hspace{1mm}}c@{\hspace{1mm}}c@{\hspace{1mm}}c@{\hspace{1mm}}}
\times & 0      & 0      & 0      & 0& 0     & 0      \\[-0.05cm]
\times & \times & \times & \times & 0     & 0& 0      \\[-0.05cm]
+      & \times & \times & \times & \times& 0     & 0\\[-0.05cm]
0      & +      & \times & \times & \times& \times& 0      \\[-0.05cm]
0      & 0      & +      & \times & \times& \times& \times \\[-0.05cm]
0      & 0      & 0      & +      & \times& \times& \times \\[-0.05cm]
0      & 0      & 0      & 0      & +     & \times& \times \\[-0.05cm]
0      & 0      & 0      & 0      & 0     & +     & \times \\[-0.05cm]
0      & 0      & 0      & 0      & 0     & 0     & \times \\[-0.05cm]
0      & 0      & 0      & 0      & 0     & 0     & \times \\[-0.05cm]
0      & 0      & 0      & 0      & 0     & 0     & \times 
\end{array}\right].
\end{align*}
Now we remove the second subdiagonal on the right hand side by adding a descending sequence of Givens rotations from the left:
\begin{align}
\begin{array}{c@{\hspace{1mm}}c@{\hspace{1mm}}c@{\hspace{1mm}}
  c@{\hspace{1mm}}c@{\hspace{1mm}}c@{\hspace{1mm}}c@{\hspace{1mm}}
  c@{\hspace{1mm}}c@{\hspace{1mm}}c@{\hspace{1mm}}c@{\hspace{1mm}}}
&   &   &   &   &   &   &   &   &   &   \\[-0.05cm]
&   &   &   &   &   &   &\Rc&   &\Rc&   \\[-0.05cm]
&   &   &   &   &   &\Rc&\rc&\Rc&\rc&\Rc\\[-0.05cm]
&   &   &   &   &\Rc&\rc&\Rc&\rc&\Rc&\rc\\[-0.05cm] 
&   &   &   &\Rc&\rc&\Rc&\rc&\Rc&\rc&   \\[-0.05cm]
&   &   &\Rc&\rc&\Rc&\rc&\Rc&\rc&   &   \\[-0.05cm] 
&   &\Rc&\rc&\Rc&\rc&\Rc&\rc&   &   &   \\[-0.05cm]
&   &\rc&\Rc&\rc&\Rc&\rc&   &   &   &   \\[-0.05cm] 
&   &   &\rc&\Rc&\rc&   &   &   &   &   \\[-0.05cm] 
&   &   &   &\rc&   &   &   &   &   &   \\[-0.05cm] 
&   &   &   &   &   &   &   &   &   & 
\end{array}
\left[\begin{array}{c@{\hspace{1mm}}c@{\hspace{1mm}}
                  c@{\hspace{1mm}}c@{\hspace{1mm}}c@{\hspace{1mm}}c@{\hspace{1mm}}c@{\hspace{1mm}}}
\times & 0& 0& 0&0& 0& 0 \\[-0.05cm]
\times & \times& 0& 0&0& 0& 0 \\[-0.05cm]
\times & \times& \times& 0&0& 0& 0 \\[-0.05cm]
\times & \times& \times& \times&0& 0& 0 \\[-0.05cm]
\times & \times& \times& \times&\times& 0& 0 \\[-0.05cm]
0&\times & \times& \times& \times&\times& 0 \\[-0.05cm]
0&0&\times & \times& \times& \times&\times\\[-0.05cm]
0&0&0 & \times& \times& \times&\times\\[-0.05cm]
0&0&0& 0& \times& \times&\times\\[-0.05cm]
0&0&0 & 0&0& \times&\times\\[-0.05cm]
0&0&0 & 0&0&0&\times
\end{array}\right]
\quad=&
\quad
\left[
\begin{array}{c@{\hspace{1mm}}c@{\hspace{1mm}}c@{\hspace{1mm}}c@{\hspace{1mm}}c@{\hspace{1mm}}c@{\hspace{1mm}}c@{\hspace{1mm}}}
\times & 0      & 0      & 0      & 0     & 0     & 0      \\[-0.05cm]
\times & \times & \times & \times & \times& 0     & 0      \\[-0.05cm]
0      & \times & \times & \times & \times& \times& 0\\[-0.05cm]
0      & 0      & \times & \times & \times& \times& \times \\[-0.05cm]
0      & 0      & 0      & \times & \times& \times& \times \\[-0.05cm]
0      & 0      & 0      & 0      & \times& \times& \times \\[-0.05cm]
0      & 0      & 0      & 0      & 0     & \times& \times \\[-0.05cm]
0      & 0      & 0      & 0      & 0     & 0     & \times \\[-0.05cm]
0      & 0      & 0      & 0      & 0     & 0     & \times \\[-0.05cm]
0      & 0      & 0      & 0      & 0     & 0     & \times \\[-0.05cm]
0      & 0      & 0      & 0      & 0     & 0     & \times 
\end{array}\right].\label{eq:qhshiftreplacesequence}
\end{align}
Finally we can reduce the last column of the right hand side thus bringing it into Hessenberg-form. This can be achieved using an ascending sequence of Givens rotations of length $2d-1$, i.e. we add one Givens rotation to each existing sequence at the end:
\begin{align}
\underbrace{\begin{array}{c@{\hspace{1mm}}c@{\hspace{1mm}}c@{\hspace{1mm}}
  c@{\hspace{1mm}}c@{\hspace{1mm}}c@{\hspace{1mm}}c@{\hspace{1mm}}
  c@{\hspace{1mm}}c@{\hspace{1mm}}c@{\hspace{1mm}}c@{\hspace{1mm}}}
&   &   &   &   &   &   &   &   &   &   \\[-0.05cm]
&   &   &   &   &   &   &\Rc&   &\Rc&   \\[-0.05cm]
&   &   &   &   &   &\Rc&\rc&\Rc&\rc&\Rc\\[-0.05cm]
&   &   &   &   &\Rc&\rc&\Rc&\rc&\Rc&\rc\\[-0.05cm] 
&   &   &   &\Rc&\rc&\Rc&\rc&\Rc&\rc&   \\[-0.05cm]
&   &   &\Rc&\rc&\Rc&\rc&\Rc&\rc&   &   \\[-0.05cm] 
&   &\Rc&\rc&\Rc&\rc&\Rc&\rc&   &   &   \\[-0.05cm]
&\Rc&\rc&\Rc&\rc&\Rc&\rc&   &   &   &   \\[-0.05cm] 
&\rc&\Rc&\rc&\Rc&\rc&   &   &   &   &   \\[-0.05cm] 
&   &\rc&\Rc&\rc&   &   &   &   &   &   \\[-0.05cm] 
&   &   &\rc&   &   &   &   &   &   & 
\end{array}}_{=:Q^2}
\left[\begin{array}{c@{\hspace{1mm}}c@{\hspace{1mm}}
                  c@{\hspace{1mm}}c@{\hspace{1mm}}c@{\hspace{1mm}}c@{\hspace{1mm}}c@{\hspace{1mm}}}
\times & 0& 0& 0&0& 0& 0 \\[-0.05cm]
\times & \times& 0& 0&0& 0& 0 \\[-0.05cm]
\times & \times& \times& 0&0& 0& 0 \\[-0.05cm]
\times & \times& \times& \times&0& 0& 0 \\[-0.05cm]
\times & \times& \times& \times&\times& 0& 0 \\[-0.05cm]
0&\times & \times& \times& \times&\times& 0 \\[-0.05cm]
0&0&\times & \times& \times& \times&\times\\[-0.05cm]
0&0&0 & \times& \times& \times&\times\\[-0.05cm]
0&0&0& 0& \times& \times&\times\\[-0.05cm]
0&0&0 & 0&0& \times&\times\\[-0.05cm]
0&0&0 & 0&0&0&\times
\end{array}\right]
\quad=
\quad
\underbrace{\left[
\begin{array}{c@{\hspace{1mm}}c@{\hspace{1mm}}c@{\hspace{1mm}}c@{\hspace{1mm}}c@{\hspace{1mm}}c@{\hspace{1mm}}c@{\hspace{1mm}}}
\times & 0      & 0      & 0      & 0     & 0     & 0      \\[-0.05cm]
\times & \times & \times & \times & \times& 0     & 0      \\[-0.05cm]
0      & \times & \times & \times & \times& \times& 0\\[-0.05cm]
0      & 0      & \times & \times & \times& \times& \times \\[-0.05cm]
0      & 0      & 0      & \times & \times& \times& \times \\[-0.05cm]
0      & 0      & 0      & 0      & \times& \times& \times \\[-0.05cm]
0      & 0      & 0      & 0      & 0     & \times& \times \\[-0.05cm]
0      & 0      & 0      & 0      & 0     & 0     & \times \\[-0.05cm]
0      & 0      & 0      & 0      & 0     & 0     & 0      \\[-0.05cm]
0      & 0      & 0      & 0      & 0     & 0     & 0      \\[-0.05cm]
0      & 0      & 0      & 0      & 0     & 0     & 0     
\end{array}\right]}_{=:H^2}\label{eq:qhshiftstep2final}
\end{align}
Notice, that $Q^2$ \eqref{eq:qhshiftstep2final} has almost the same structure as $Q^1$ \eqref{eq:qhshiftstep1} except for two additional rotations on the second and third row.
\pagebreak  

\textbf{Step 3:}
We start as we have in the second step by shifting the factorization $Q^2A^2=H^2$ one entry to the top left. Corresponding to \eqref{eq:qhshiftstep2} we get
\begin{align*}
 \hat{Q}^3A^3=\begin{array}{c@{\hspace{1mm}}c@{\hspace{1mm}}c@{\hspace{1mm}}
  c@{\hspace{1mm}}c@{\hspace{1mm}}c@{\hspace{1mm}}c@{\hspace{1mm}}
  c@{\hspace{1mm}}c@{\hspace{1mm}}c@{\hspace{1mm}}c@{\hspace{1mm}}}
   &   &   &   &   &   &   &\Rc&   &\Rc&   \\[-0.05cm]
   &   &   &   &   &   &\Rc&\rc&\Rc&\rc&\Rc\\[-0.05cm]
   &   &   &   &   &\Rc&\rc&\Rc&\rc&\Rc&\rc\\[-0.05cm] 
   &   &   &   &\Rc&\rc&\Rc&\rc&\Rc&\rc&   \\[-0.05cm]
   &   &   &\Rc&\rc&\Rc&\rc&\Rc&\rc&   &   \\[-0.05cm] 
   &   &\Rc&\rc&\Rc&\rc&\Rc&\rc&   &   &   \\[-0.05cm]
   &\Rc&\rc&\Rc&\rc&\Rc&\rc&   &   &   &   \\[-0.05cm] 
\Rc&\rc&\Rc&\rc&\Rc&\rc&   &   &   &   &   \\[-0.05cm] 
\rc&\Rc&\rc&\Rc&\rc&   &   &   &   &   &   \\[-0.05cm] 
   &\rc&\Rc&\rc&   &   &   &   &   &   &   \\[-0.05cm] 
   &   &\rc&   &   &   &   &   &   &   &
\end{array}
\left[\begin{array}{c@{\hspace{1mm}}c@{\hspace{1mm}}
                  c@{\hspace{1mm}}c@{\hspace{1mm}}c@{\hspace{1mm}}c@{\hspace{1mm}}c@{\hspace{1mm}}}
\times & 0& 0& 0&0& 0& 0 \\[-0.05cm]
\times & \times& 0& 0&0& 0& 0 \\[-0.05cm]
\times & \times& \times& 0&0& 0& 0 \\[-0.05cm]
\times & \times& \times& \times&0& 0& 0 \\[-0.05cm]
\times & \times& \times& \times&\times& 0& 0 \\[-0.05cm]
0&\times & \times& \times& \times&\times& 0 \\[-0.05cm]
0&0&\times & \times& \times& \times&\times\\[-0.05cm]
0&0&0 & \times& \times& \times&\times\\[-0.05cm]
0&0&0& 0& \times& \times&\times\\[-0.05cm]
0&0&0 & 0&0& \times&\times\\[-0.05cm]
0&0&0 & 0&0&0&\times
\end{array}\right]
\quad=
\quad
\left[
\begin{array}{c@{\hspace{1mm}}c@{\hspace{1mm}}c@{\hspace{1mm}}c@{\hspace{1mm}}c@{\hspace{1mm}}c@{\hspace{1mm}}c@{\hspace{1mm}}}
\times & \times & \times & \times & 0& 0     & 0      \\[-0.05cm]
\times & \times & \times & \times & \times& 0& 0      \\[-0.05cm]
0      & \times & \times & \times & \times& \times& 0\\[-0.05cm]
0      & 0      & \times & \times & \times& \times& \times \\[-0.05cm]
0      & 0      & 0      & \times & \times& \times& \times \\[-0.05cm]
0      & 0      & 0      & 0      & \times& \times& \times \\[-0.05cm]
0      & 0      & 0      & 0      & 0     & \times& \times \\[-0.05cm]
0      & 0      & 0      & 0      & 0     & 0     & \times \\[-0.05cm]
0      & 0      & 0      & 0      & 0     & 0     & \times \\[-0.05cm]
0      & 0      & 0      & 0      & 0     & 0     & \times \\[-0.05cm]
0      & 0      & 0      & 0      & 0     & 0     & \times
\end{array}\right]=\tilde{H}^3
\end{align*}
Analogous to step 2 we want to remove all Givens rotations acting on the first row so that we can apply the shift again. This time however we have two descending sequences starting in the first row. We could of course remove both outer-most sequences of Givens-rotations and add them again starting in the second row, but in our illustrative example this would already cost more than simply restarting an entire factorization from scratch. One may argue that for higher values $d$ this would not be the case. However when taking a closer look at our sequences of Givens-rotations we can see that if we continue this procedure we would have $3$ rotations acting on the first row when we arrive at the 4th step and so on, up to $2d-1$, which is the total number of descending sequences. We therefore have to solve this problem another way.\\
If we apply Theorem \ref{thm:qhtransform} below to the two outer-most sequences, we would arrive at
\begin{align}
 \begin{array}{c@{\hspace{1mm}}c@{\hspace{1mm}}c@{\hspace{1mm}}
  c@{\hspace{1mm}}c@{\hspace{1mm}}c@{\hspace{1mm}}c@{\hspace{1mm}}
  c@{\hspace{1mm}}c@{\hspace{1mm}}c@{\hspace{1mm}}c@{\hspace{1mm}}}
   &   &   &   &   &   &   &\Rc&   &\Rc&   \\[-0.05cm]
   &   &   &   &   &   &\Rc&\rc&\Rc&\rc&\Rc\\[-0.05cm]
   &   &   &   &   &\Rc&\rc&\Rc&\rc&\Rc&\rc\\[-0.05cm] 
   &   &   &   &\Rc&\rc&\Rc&\rc&\Rc&\rc&   \\[-0.05cm]
   &   &   &\Rc&\rc&\Rc&\rc&\Rc&\rc&   &   \\[-0.05cm] 
   &   &\Rc&\rc&\Rc&\rc&\Rc&\rc&   &   &   \\[-0.05cm]
   &\Rc&\rc&\Rc&\rc&\Rc&\rc&   &   &   &   \\[-0.05cm] 
\Rc&\rc&\Rc&\rc&\Rc&\rc&   &   &   &   &   \\[-0.05cm] 
\rc&\Rc&\rc&\Rc&\rc&   &   &   &   &   &   \\[-0.05cm] 
   &\rc&\Rc&\rc&   &   &   &   &   &   &   \\[-0.05cm] 
   &   &\rc&   &   &   &   &   &   &   &
\end{array}\rightarrow
 \begin{array}{c@{\hspace{1mm}}c@{\hspace{1mm}}c@{\hspace{1mm}}c@{\hspace{1mm}}
  c@{\hspace{1mm}}c@{\hspace{1mm}}c@{\hspace{1mm}}c@{\hspace{1mm}}
  c@{\hspace{1mm}}c@{\hspace{1mm}}c@{\hspace{1mm}}c@{\hspace{1mm}}}
   &   &   &   &   &   &   &   &\Rc&   &   &   \\[-0.05cm]
   &   &   &   &   &   &   &\Rc&\rc&\Rc&   &\Rc\\[-0.05cm]
   &   &   &   &   &   &\Rc&\rc&\Rc&\rc&\Rc&\rc\\[-0.05cm] 
   &   &   &   &   &\Rc&\rc&\Rc&\rc&\Rc&\rc&   \\[-0.05cm]
   &   &   &   &\Rc&\rc&\Rc&\rc&\Rc&\rc&   &   \\[-0.05cm] 
   &   &   &\Rc&\rc&\Rc&\rc&\Rc&\rc&   &   &   \\[-0.05cm]
   &   &\Rc&\rc&\Rc&\rc&\Rc&\rc&   &   &   &   \\[-0.05cm] 
   &\Rc&\rc&\Rc&\rc&\Rc&\rc&   &   &   &   &   \\[-0.05cm] 
\Rc&\rc&\Rc&\rc&\Rc&\rc&   &   &   &   &   &   \\[-0.05cm] 
\rc&   &\rc&\Rc&\rc&   &   &   &   &   &   &   \\[-0.05cm] 
   &   &   &\rc&   &   &   &   &   &   &   &
\end{array},\label{eq:qhshiftreorder}
\end{align}
i.e. all but one Givens rotation acting on the first row has been moved to the end of the outer-most sequence. Note, that when applying Theorem \ref{thm:qhtransform}, the values $(c,s)$ of all rotations involved will generally change and the application costs $\O(n)$ flops for each rotation that has been removed in the first row.\\
We can now remove the remaining rotation acting on the first row as we have in step 2.

\textbf{Step \boldmath$4,\ldots,2d$}:
We repeat the procedure applied in step $3$, only this time Theorem \ref{thm:qhtransform} has to be applied to $3$ descending sequences. With each step the number of Givens-rotations acting on the first row increases by one, therefore the number of descending sequences to which we apply Theorem \ref{thm:qhtransform} also increases by one with each step. This number is however limited by the total number of descending sequences, $2d-1$, and therefore the effort required by applying Theorem \ref{thm:qhtransform} is only $\O(nd)$ flops.

\textbf{Step \boldmath$2d+1,\ldots,k_{\max}$}:
From now on the entire procedure simply repeats itself.

The whole procedure is summarized in Algorithm \ref{alg:qhshift}.\\
\begin{algorithm}[H]
\caption{QH-Shift Algorithm}
\label{alg:qhshift}
\LinesNumbered
\SetAlgoLined
\KwIn{A sequence of $d$-banded consecutive matrices $\{A^{k}\}_{k=1,\ldots,k_{\max}}$ as in \eqref{eq:Ankshift}}
\KwOut{A sequence of upper triangular matrices $\{R^{k}\}_{k=1,\ldots,k_{\max}}$ with bandwidth $d$}
First step: Compute QH factorization $Q^1A^{1}=H^1$ using $2d-1$ sequences of Givens rotations as in \eqref{eq:qhshiftstep1}\;
Compute QR factorization ${G}^1Q^1A^{1}=R^1$ using one more sequence of Givens rotations\;
\For{$k=2,\ldots,k_{\max}$}
{
Shift factorization $Q^{k-1}A^{k-1}=H^{k-1}\rightarrow \hat{Q}^{k}A^{k}=\tilde{H}^{k}$ as in \eqref{eq:qhshiftstep2}\;
\label{alg:qhshift:transform}Move rotations acting on the first row to the left-most sequence as in Theorem \ref{thm:qhtransform}\;
Remove the last rotation acting on the first row by replacing the left-most sequence as in \eqref{eq:qhshiftreplacesequence}\;
Bring $\hat{H}^k$ to Hessenberg form by removing $2d-1$ entries in the last column as in \eqref{eq:qhshiftstep2final}\;
Compute QR factorization ${G}^kQ^kA^{k}=R^k$\;
}
\end{algorithm}
The reordering of Givens rotations applied in \eqref{eq:qhshiftreorder} is described in the following theorem:
\begin{theorem}
\label{thm:qhtransform}
Let $l,s,m\in\N$, $l+s\leq m$ and let
\begin{align*}
Q=(G_lG_{l-1}\cdots G_1)(G_{l+1}G_l\cdots G_1)\cdots(G_{l+s-1}G_{l+s-2}\cdots G_1)\in\C^{m\times m}
\end{align*}
be a product of $s$ descending sequences of Givens rotations, each starting in the first row and decreasing in length (from left to right).\\
Then $Q$ can be described as a product of $s$ sequences of Givens rotations of the form
\begin{align*}
 Q=(G_{l+s-1}G_{l+s-2}\cdots G_1)(G_{l+1}G_l\cdots G_2)(G_{l+2}G_{l+1}\cdots G_2)\cdots (G_{l+s-1}G_{l+s-2}\cdots G_2)
\end{align*}
\begin{align*}
Q=\begin{array}{c@{\hspace{1mm}}c@{\hspace{1mm}}c@{\hspace{1mm}}c@{\hspace{1mm}}c@{\hspace{1mm}}c@{\hspace{1mm}}
  c@{\hspace{1mm}}c@{\hspace{1mm}}c@{\hspace{1mm}}c@{\hspace{1mm}}
  c@{\hspace{1mm}}c@{\hspace{1mm}}c@{\hspace{1mm}}c@{\hspace{1mm}}c@{\hspace{1mm}}c@{\hspace{1mm}}}
   &   &   &   &   &   &   &\Rc&   &\Rc&   &\Rc&   &\Rc&   &\Rc\\[-0.05cm]
   &   &   &   &   &   &\Rc&\rc&\Rc&\rc&\Rc&\rc&\Rc&\rc&\Rc&\rc\\[-0.05cm] 
   &   &   &   &   &\Rc&\rc&\Rc&\rc&\Rc&\rc&\Rc&\rc&\Rc&\rc&   \\[-0.05cm]
   &   &   &   &\Rc&\rc&\Rc&\rc&\Rc&\rc&\Rc&\rc&\Rc&\rc&   &   \\[-0.05cm] 
   &   &   &\Rc&\rc&\Rc&\rc&\Rc&\rc&\Rc&\rc&\Rc&\rc&   &   &   \\[-0.05cm]
   &   &\Rc&\rc&\Rc&\rc&\Rc&\rc&\Rc&\rc&\Rc&\rc&   &   &   &   \\[-0.05cm] 
   &   &\rc&\Rc&\rc&\Rc&\rc&\Rc&\rc&\Rc&\rc&   &   &   &   &   \\[-0.05cm] 
   &   &   &\rc&\Rc&\rc&\Rc&\rc&\Rc&\rc&   &   &   &   &   &   \\[-0.05cm] 
   &   &   &   &\rc&\Rc&\rc&\Rc&\rc&   &   &   &   &   &   &   \\[-0.05cm] 
   &   &   &   &   &\rc&\Rc&\rc&   &   &   &   &   &   &   &   \\[-0.05cm] 
   &   &   &   &   &   &\rc&   &   &   &   &   &   &   &   &
\end{array}
=
\begin{array}{c@{\hspace{1mm}}c@{\hspace{1mm}}c@{\hspace{1mm}}c@{\hspace{1mm}}c@{\hspace{1mm}}c@{\hspace{1mm}}c@{\hspace{1mm}}c@{\hspace{1mm}}
  c@{\hspace{1mm}}c@{\hspace{1mm}}c@{\hspace{1mm}}c@{\hspace{1mm}}
  c@{\hspace{1mm}}c@{\hspace{1mm}}c@{\hspace{1mm}}c@{\hspace{1mm}}c@{\hspace{1mm}}c@{\hspace{1mm}}}
   &   &   &   &   &   &   &   &   &\Rc&   &   &   &   &   &   &   &   \\[-0.05cm]
   &   &   &   &   &   &   &   &\Rc&\rc&\Rc&   &\Rc&   &\Rc&   &\Rc&   \\[-0.05cm] 
   &   &   &   &   &   &   &\Rc&\rc&\Rc&\rc&\Rc&\rc&\Rc&\rc&\Rc&\rc&   \\[-0.05cm]
   &   &   &   &   &   &\Rc&\rc&\Rc&\rc&\Rc&\rc&\Rc&\rc&\Rc&\rc&   &   \\[-0.05cm] 
   &   &   &   &   &\Rc&\rc&\Rc&\rc&\Rc&\rc&\Rc&\rc&\Rc&\rc&   &   &   \\[-0.05cm]
   &   &   &   &\Rc&\rc&\Rc&\rc&\Rc&\rc&\Rc&\rc&\Rc&\rc&   &   &   &   \\[-0.05cm] 
   &   &   &\Rc&\rc&\Rc&\rc&\Rc&\rc&\Rc&\rc&\Rc&\rc&   &   &   &   &   \\[-0.05cm] 
   &   &\Rc&\rc&   &\rc&\Rc&\rc&\Rc&\rc&\Rc&\rc&   &   &   &   &   &   \\[-0.05cm] 
   &\Rc&\rc&   &   &   &\rc&\Rc&\rc&\Rc&\rc&   &   &   &   &   &   &   \\[-0.05cm] 
\Rc&\rc&   &   &   &   &   &\rc&\Rc&\rc&   &   &   &   &   &   &   &   \\[-0.05cm] 
\rc&   &   &   &   &   &   &   &\rc&   &   &   &   &   &   &   &   &
\end{array}
\end{align*}
\end{theorem}
\begin{proof}
We start by noting the so-called {\sl shift-through lemma} of \cite{MatrixComp08} (see Lemma 9.38 there), which states that a product of $3$ Givens rotations of the form $G_2G_1G_2$ can be transformed into $3$ Givens rotations of the form $G_1G_2G_1$ and vice versa, i.e.
\begin{align*}
Q=\begin{array}{c@{\hspace{1mm}}c@{\hspace{1mm}}c@{\hspace{1mm}}}
&\Rc&      \\[-0.05cm] 
\Rc&\rc&\Rc      \\[-0.05cm]
\rc&&\rc
\end{array}=
\begin{array}{c@{\hspace{1mm}}c@{\hspace{1mm}}c@{\hspace{1mm}}}
\Rc&&\Rc\\[-0.05cm] 
\rc&\Rc&\rc      \\[-0.05cm]
&\rc&
\end{array}.
\end{align*}
Of course the values $c,s$ of all rotations involved change. Note that this only holds for products without intermediate rotations, e.g. it could not be applied to 
\begin{align*}
Q=\begin{array}{c@{\hspace{1mm}}c@{\hspace{1mm}}c@{\hspace{1mm}}}
&\Rc&\\[-0.05cm] 
\Rc&\rc&\Rc\\[-0.05cm] 
\rc&\Rc&\rc      \\[-0.05cm]
&\rc&
\end{array}.
\end{align*}
As described in \cite{implicitDoubleShiftQRalgorithm} a repeated application of this lemma to $2$ descending sequences of Givens rotations $(G_lG_{l-1}\cdots G_1)(G_{l+1}G_l\cdots G_1)$ leads to the {\sl shift-through lemma of higher length}:
\begin{align*}
(G_l\cdots G_1)(G_{l+1}\cdots G_1)=&(G_{l+1}\cdots G_1)(G_{l+1}\cdots G_2)\\
\begin{array}{c@{\hspace{1mm}}c@{\hspace{1mm}}c@{\hspace{1mm}}c@{\hspace{1mm}}c@{\hspace{1mm}}c@{\hspace{1mm}}
  c@{\hspace{1mm}}c@{\hspace{1mm}}c@{\hspace{1mm}}c@{\hspace{1mm}}
  c@{\hspace{1mm}}c@{\hspace{1mm}}}
   &   &   &   &   &   &   &   &   &\Rc&   &\Rc\\[-0.05cm]
   &   &   &   &   &   &   &   &\Rc&\rc&\Rc&\rc\\[-0.05cm] 
   &   &   &   &   &   &   &\Rc&\rc&\Rc&\rc&   \\[-0.05cm]
   &   &   &   &   &   &\Rc&\rc&\Rc&\rc&   &   \\[-0.05cm] 
   &   &   &   &   &\Rc&\rc&\Rc&\rc&   &   &   \\[-0.05cm]
   &   &   &   &   &\rc&\Rc&\rc&   &   &   &   \\[-0.05cm] 
   &   &   &   &   &   &\rc&   &   &   &   &
\end{array}=\begin{array}{c@{\hspace{1mm}}c@{\hspace{1mm}}c@{\hspace{1mm}}c@{\hspace{1mm}}c@{\hspace{1mm}}c@{\hspace{1mm}}
  c@{\hspace{1mm}}c@{\hspace{1mm}}c@{\hspace{1mm}}c@{\hspace{1mm}}
  c@{\hspace{1mm}}c@{\hspace{1mm}}}
   &   &   &   &   &   &   &   &   &   &\Rc&\\[-0.05cm]
   &   &   &   &   &   &   &\Rc&   &\Rc&\rc&\Rc\\[-0.05cm] 
   &   &   &   &   &   &\Rc&\rc&\Rc&\rc&   &\rc\\[-0.05cm]
   &   &   &   &   &\Rc&\rc&\Rc&\rc&   &   &\\[-0.05cm] 
   &   &   &   &\Rc&\rc&\Rc&\rc&   &   &   &\\[-0.05cm]
   &   &   &   &\rc&\Rc&\rc&   &   &   &   &\\[-0.05cm] 
   &   &   &   &   &\rc&   &   &   &   &   &
\end{array}=&
\begin{array}{c@{\hspace{1mm}}c@{\hspace{1mm}}c@{\hspace{1mm}}c@{\hspace{1mm}}c@{\hspace{1mm}}c@{\hspace{1mm}}
  c@{\hspace{1mm}}c@{\hspace{1mm}}c@{\hspace{1mm}}c@{\hspace{1mm}}
  c@{\hspace{1mm}}c@{\hspace{1mm}}}
   &   &   &   &   &   &   &   &   &   &\Rc&\\[-0.05cm]
   &   &   &   &   &   &   &   &   &\Rc&\rc&\Rc\\[-0.05cm] 
   &   &   &   &   &   &\Rc&   &\Rc&\rc&\Rc&\rc\\[-0.05cm]
   &   &   &   &   &\Rc&\rc&\Rc&\rc&   &\rc&\\[-0.05cm] 
   &   &   &   &\Rc&\rc&\Rc&\rc&   &   &   &\\[-0.05cm]
   &   &   &   &\rc&\Rc&\rc&   &   &   &   &\\[-0.05cm] 
   &   &   &   &   &\rc&   &   &   &   &   &
\end{array}=\cdots=
 \begin{array}{c@{\hspace{1mm}}c@{\hspace{1mm}}c@{\hspace{1mm}}c@{\hspace{1mm}}c@{\hspace{1mm}}c@{\hspace{1mm}}
  c@{\hspace{1mm}}c@{\hspace{1mm}}c@{\hspace{1mm}}c@{\hspace{1mm}}
  c@{\hspace{1mm}}c@{\hspace{1mm}}}
   &   &   &   &   &   &   &   &   &\Rc&   &   \\[-0.05cm]
   &   &   &   &   &   &   &   &\Rc&\rc&\Rc&   \\[-0.05cm] 
   &   &   &   &   &   &   &\Rc&\rc&\Rc&\rc&   \\[-0.05cm]
   &   &   &   &   &   &\Rc&\rc&\Rc&\rc&   &   \\[-0.05cm] 
   &   &   &   &   &\Rc&\rc&\Rc&\rc&   &   &   \\[-0.05cm]
   &   &   &   &\Rc&\rc&\Rc&\rc&   &   &   &   \\[-0.05cm] 
   &   &   &   &\rc&   &\rc&   &   &   &   &
\end{array}
\end{align*}
Figuratively speaking we have moved the rotation from the top right to the lower left. If the right sequence is of higher length than the left sequence, the additional Givens rotations will be added to the left sequence in the end, i.e.
\begin{align*}
 (G_lG_{l-1}\cdots G_1)(G_{l+t}G_{l+t-1}\cdots G_1)=&(G_{l+t}G_{l+t-1}\cdots G_1)(G_{l+1}G_{l}\cdots G_2)\\
 \begin{array}{c@{\hspace{1mm}}c@{\hspace{1mm}}c@{\hspace{1mm}}c@{\hspace{1mm}}c@{\hspace{1mm}}c@{\hspace{1mm}}
  c@{\hspace{1mm}}c@{\hspace{1mm}}c@{\hspace{1mm}}c@{\hspace{1mm}}
  c@{\hspace{1mm}}c@{\hspace{1mm}}}
   &   &   &   &   &   &   &   &   &\Rc&   &\Rc\\[-0.05cm]
   &   &   &   &   &   &   &   &\Rc&\rc&\Rc&\rc\\[-0.05cm] 
   &   &   &   &   &   &   &\Rc&\rc&\Rc&\rc&   \\[-0.05cm]
   &   &   &   &   &   &\Rc&\rc&\Rc&\rc&   &   \\[-0.05cm] 
   &   &   &   &   &\Rc&\rc&\Rc&\rc&   &   &   \\[-0.05cm]
   &   &   &   &   &\rc&\Rc&\rc&   &   &   &   \\[-0.05cm] 
   &   &   &   &   &\Rc&\rc&   &   &   &   &   \\[-0.05cm]
   &   &   &   &\Rc&\rc&   &   &   &   &   &   \\[-0.05cm]
   &   &   &\Rc&\rc&   &   &   &   &   &   &   \\[-0.05cm]
   &   &   &\rc&   &   &   &   &   &   &   &
   \end{array}
=&\cdots=
 \begin{array}{c@{\hspace{1mm}}c@{\hspace{1mm}}c@{\hspace{1mm}}c@{\hspace{1mm}}c@{\hspace{1mm}}c@{\hspace{1mm}}
  c@{\hspace{1mm}}c@{\hspace{1mm}}c@{\hspace{1mm}}c@{\hspace{1mm}}
  c@{\hspace{1mm}}c@{\hspace{1mm}}}
   &   &   &   &   &   &   &   &   &\Rc&   &   \\[-0.05cm]
   &   &   &   &   &   &   &   &\Rc&\rc&\Rc&   \\[-0.05cm] 
   &   &   &   &   &   &   &\Rc&\rc&\Rc&\rc&   \\[-0.05cm]
   &   &   &   &   &   &\Rc&\rc&\Rc&\rc&   &   \\[-0.05cm] 
   &   &   &   &   &\Rc&\rc&\Rc&\rc&   &   &   \\[-0.05cm]
   &   &   &   &\Rc&\rc&\Rc&\rc&   &   &   &   \\[-0.05cm] 
   &   &   &\Rc&\rc&   &\rc&   &   &   &   &   \\[-0.05cm]
   &   &\Rc&\rc&   &   &   &   &   &   &   &   \\[-0.05cm]
   &\Rc&\rc&   &   &   &   &   &   &   &   &   \\[-0.05cm]
   &\rc&   &   &   &   &   &   &   &   &   &      
\end{array}
\end{align*}
The theorem follows directly from applying the shift-through-lemma of higher length pairwise $s-1$ times to the descending sequences from the right to the left.
\end{proof}

\textbf{Restarted QH-Shift.}
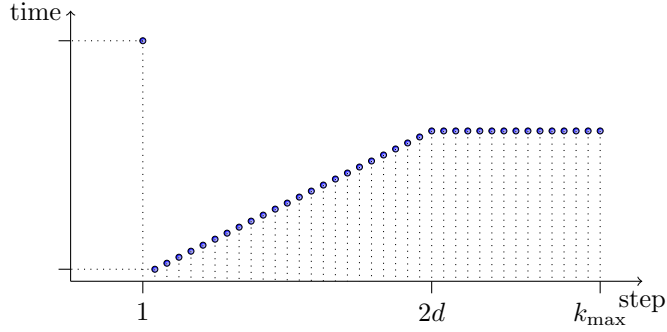
\begin{figure}
\begin{center}
\begin{tikzpicture}[scale=0.8]

\draw[->] (0,0) -- (9.5,0) node[anchor=north] {step};
\draw[->] (0,0) -- (0,4.5) node[anchor=east] {time};

\draw (1.2,0)--(1.2,-0.2) node[anchor=north] {$1$};
\draw (6.0,0)--(6.0,-0.2) node[anchor=north] {$2d$};
\draw (8.8,0)--(8.8,-0.2) node[anchor=north] {$k_{\max}$};

\draw (0,0.2)--(-0.2,0.2) node[left] {};
\draw (0,4)--(-0.2,4) node[left] {};

\draw[smallball blue](1.2,4) circle (0.05);
\draw[dotted](1.2,4)--(1.2,0);
\draw[dotted](1.2,4)--(0,4);
\draw[dotted](1.4,0.2)--(0,0.2);
\foreach \i in {1.4,1.6,...,6.2}
{qhshiftonereset
        \draw[smallball blue](\i,\i/2-0.5) circle (0.05);
        \draw[dotted](\i,\i/2-0.5)--(\i,0);
}
\foreach \i in {6.2,6.4,...,8.8}
{
        \draw[smallball blue](\i,2.5) circle (0.05);
        \draw[dotted](\i,2.5)--(\i,0);
}	
\end{tikzpicture}
\caption{QH-Shift Algorithm without reset: Time per step}
\label{fig:qhshiftnoreset}
\end{center}
\end{figure}
As stated before, the number of descending Givens sequences to which Theorem \ref{thm:qhtransform} is applied in Algorithm \ref{alg:qhshift} 
grows with each step. Therefore the algorithm is fastest in the second step and then slows down until it reaches step $2d$, as is illustrated 
in Figure \ref{fig:qhshiftnoreset}. Depending on the time required for the initial QH factorization in step 1 and the time required in the 
following steps, it may be more efficient to ``restart'' the method after step number $r$ (for some $r\in\{2,...,2d\}$) in order to take advantage 
of the cheap steps with number $2,...,r$. This is illustrated in Figure \ref{fig:qhshiftonereset}. The time parameters required to determine the 
optimal point $r$ for a restart can be estimated during runtime (see e.g. Figure \ref{fig:impureLaurent2} below).
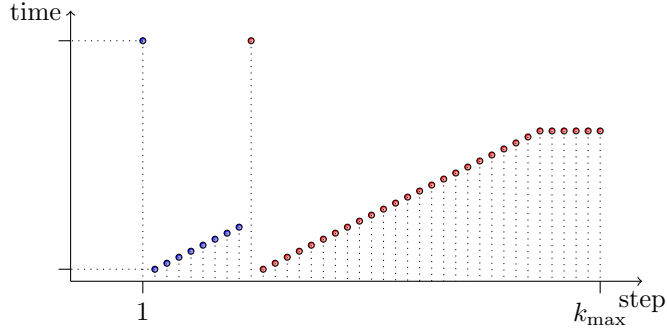
\begin{figure}
\begin{center}
\begin{tikzpicture}[scale=0.8]

\draw[->] (0,0) -- (9.5,0) node[anchor=north] {step};
\draw[->] (0,0) -- (0,4.5) node[anchor=east] {time};

\draw (1.2,0)--(1.2,-0.2) node[anchor=north] {$1$};
\draw (8.8,0)--(8.8,-0.2) node[anchor=north] {$k_{\max}$};

\draw (0,0.2)--(-0.2,0.2) node[left] {};
\draw (0,4)--(-0.2,4) node[left] {};

\draw[smallball blue](1.2,4) circle (0.05);
\draw[dotted](1.2,4)--(1.2,0);
\draw[dotted](1.2,4)--(0,4);
\draw[dotted](1.4,0.2)--(0,0.2);

\foreach \i in {1.4,1.6,...,3}
{
        \draw[smallball blue](\i,\i/2-0.5) circle (0.05);
        \draw[dotted](\i,\i/2-0.5)--(\i,0);
}

\draw[smallball red](3,4) circle (0.05);
\draw[dotted](3,4)--(3,0);
\foreach \i in {1.4,1.6,...,6.2}
{
        \draw[smallball red](\i+1.8,\i/2-0.5) circle (0.05);
        \draw[dotted](\i+1.8,\i/2-0.5)--(\i+1.8,0);
}

\foreach \i in {6.2,6.4,...,7}
{
        \draw[smallball red](\i+1.8,2.5) circle (0.05);
        \draw[dotted](\i+1.8,2.5)--(\i+1.8,0);
}	
\end{tikzpicture}
\caption{QH-Shift Algorithm with one restart after $r=9$ steps.}
\label{fig:qhshiftonereset}
\end{center}
\end{figure}

\section{Applications}
\label{sec:applications}
\subsection{Laurent Operators with local impurities}
\label{ssec:laurentOperatorsImpurities}
We start with bi-infinite matrices with constant diagonals, also known as Laurent operators. Let $a$ be a continuous function on the complex unit circle $\T$, and denote by $(a_j)_{j\in\Z}$ the sequence of Fourier coefficients of $a$ so that
\begin{align}
 a(t)=\sum_{j\in\Z} a_j t^j,\quad t=e^{i\theta}\in\T.\label{eq:laurentsymbol}
\end{align}
We denote the corresponding bounded linear operator on $\ell^2:=\ell^2(\Z)$ (see \cite{BoeSi2}), as well as the infinite matrix
\begin{align*}
\left(
\begin{array}{ccccccc}
\ddots&\ddots&\ddots&\ddots&\ddots&\ddots&\ddots\\
\ddots&a_0&a_{-1}&a_{-2}&a_{-3}&a_{-4}&\ddots\\
\ddots&a_1&a_0&a_{-1}&a_{-2}&a_{-3}&\ddots\\
\ddots&a_2&a_1&a_0&a_{-1}&a_{-2}&\ddots\\
\ddots&a_3&a_2&a_1&a_0&a_{-1}&\ddots\\
\ddots&a_4&a_3&a_2&a_1&a_0&\ddots\\
\ddots&\ddots&\ddots&\ddots&\ddots&\ddots&\ddots
\end{array}\right),
\end{align*}
by $L(a)$.  Via the Fourier transform, the convolution operator $L(a)$ corresponds to multiplication on $L^2(\T)$ by the function $a$ from \eqref{eq:laurentsymbol}, which is referred to as the symbol of the operator $L(a)$. In particular, the spectrum of $L(a)$ is the image of $\T$ under the function $a$. Moreover, Laurent operators are normal and therefore $\speps(L(a))=\spec(L(a))+\epsilon\D$ can be explicitly computed. We lose these properties when we add so-called local impurities, meaning operators $E:\ell^2\rightarrow \ell^2$ with a finitely supported matrix. In condensed matter physics this corresponds to a local impurity in an otherwise periodic crystal structure. The resulting operator $L(a)+E$ is in general non-normal, and its spectrum and pseudospectra are difficult to approximate (see e.g. \cite{BoeEmLi,BoeEmSok02}). So let us apply our algorithm.

In this example we consider exponentially decreasing Fourier coefficients $a_k$ as $k\rightarrow \pm \infty$, so that $L(a)$ is in the Wiener algebra $\W$ (see Section \ref{ssec:banddominatedtobanded}). More precisely, when approximating $L(a)$ by operators with finite bandwidth, say $L_d(a)$ with bandwidth $d\in\N$, we can explicitly give upper bounds on the approximation error as in \eqref{eq:Werror}. Here this means
\begin{align}\label{eq:werrorexample}
\norm{L(a)-L_d(a)}\leq\lW L(a)-L(b)\rW=\sum_{\abs{k}>d}\abs{a_k}\leq\eta_d
\end{align}
for some error $\eta_d>0$. 
For simplicity, we choose $E$ to be of bandwidth $\leq d$, although impurities with larger support could be treated analogously with an appropriate error estimation in \eqref{eq:werrorexample}.
The resulting lower and upper bounds on $\speps(L(a)+E)$ for a given bandwidth $d$ and cut-size $N$ can be summed up, using \eqref{eq:spepsAB} and \eqref{eq:sandwich_inters}, as follows:
\begin{align}
 \bigcup_{c=0}^{d-1} \Gamma_{\eps-\eta_d}^{n,c}(L_d(a)+E)\ \subset\  \speps(L(a)+E)\ \subset\ \bigcap_{c=0}^{d-1} \Gamma_{\eps+\eta_d+\delta_N}^{n,c}(L_d(A)+E),\label{eq:laurent_specincl}
\end{align}
where $\eps>0$ and $\delta_N\in\O(\frac1N)$ denotes the approximation error introduced in \eqref{eq:epsn_block}. In order to compute the spectral
inclusion sets $\Gamma^{n,c}_{\eps}(L_d(a)+E)$ from \eqref{eq:defGammaM}, we can make use of the fact that we only need to consider a finite number
(growing linearly with $n$) of positions $k$, since $L(a)+E$ is constant along its diagonals as we move away from the support of $E$.

We rewrite the pseudospectral sub- and supersets from \eqref{eq:laurent_specincl} as
\begin{align}
 &\left\{\lambda\in\C:\; F_l(\lambda)<\eps-\eta_d\right\}\text{ and}\label{eq:laurent_specsub}\\
 &\left\{\lambda\in\C:\; F_u(\lambda)<\eps+\eta_d+\delta_N\right\}\label{eq:laurent_specsup},
\end{align}
respectively, where 
\begin{align*}
 F_l(\lambda):=&\min_{c=0,\ldots,d-1}\left(\min_{k\in c+d\Z}\left(\min(\nu((A-\lambda I)|_{\ell^2(J^n_k)}),\nu((A-\lambda I)^*|_{\ell^2(J^n_k)}))\right)\right)\\
 F_u(\lambda):=&\max_{c=0,\ldots,d-1}\left(\min_{k\in c+d\Z}\left(\min(\nu((A-\lambda I)|_{\ell^2(J^n_k)}),\nu((A-\lambda I)^*|_{\ell^2(J^n_k)}))\right)\right)
\end{align*}
with $A:=L_d(a)+E$. 
If we are only interested in $\speps(A)$ for a few values $\eps>0$, then we can approximate connected components of these sets by determining the boundary of each set. To that end, we use continuation methods to determine all $\lambda\in\C$ which satisfy
\begin{align}
  F_l(\lambda)-(\eps-\eta_d)&=0,\label{eq:boundaryFuncl}\\
  F_u(\lambda)-(\eps+\eta_d+\delta_n)&=0,\label{eq:boundaryFuncu}
\end{align}
respectively.
There have been several approaches to computing the boundary curves of pseudospectra of finite square matrices using gradient-based methods, see e.g.~\cite{Bruehl:CurveTracing, Bekas:Cobra}. However, since $F_u$ and $F_l$ involve several nested minima and maxima of smallest singular values of rectangular matrices, they are non-smooth, so that these methods are not well suited to our case.

In \cite{Mezher:PAT} a piecewise linear (PL)-continuation method was used to approximate pseudospectral boundaries of matrices, and this algorithm can be easily modified to be applied to $F_u$ and $F_l$ as well. The idea is to triangulate the complex plane and determine all triangles in which the signs of $F_u$ (respectively $F_l$) are not equal on all three vertices. The functions have therefore to be evaluated on these vertices only, and each function evaluation consists of two applications of the QH-shift method (one for $(A-\lambda I)$ and one for $(A-\lambda I)^*$). We note the difficulty of finding a starting point, an initial triangle, which can be solved using a coarse grid or a bisection based method (see \cite{Mezher:PAT}). The boundaries can be estimated prior using the coarse upper bound $\speps(L(a)+E)\subset\spec(L(a))+(\eps+\norm{E})\D$.

We applied this method to a Laurent operator defined by the coefficients 
\begin{align}
a_k:=\left\{
\begin{array}{ll}
\left(\frac{1}{2}\right)^k+1.1\left(\frac{1}{2i}\right)^k,&k>0\\
3.1, &k=0\\
\left(\frac{i}{2}\right)^{-k},&k<0
\end{array}
\right.
\label{eq:impureLaurentsymbol}
\end{align}
and an impurity $E$ which is a scaled and shifted $10\times 10$ Grcar-matrix. We refer to $L(a)$ as ``the fish'', motivated by the shape of its
spectrum. The approximation error \eqref{eq:werrorexample} can be estimated as $\eta_d\leq\frac{1}{2^{d-2}}$. The results can be seen in Figure
\ref{fig:impureLautent}. In addition to the upper and lower bounds on $\speps(L(a)+E)$, we see that the superset does not contain the origin,
implying that this particular operator is invertible. Furthermore we compare the speed of the QH-shift method, the restarted QH-shift method
and the classic QR-decomposition, including the SVD, for a single $\lambda\in\C$, as can be seen in Figure \ref{fig:impureLaurent2}.
\begin{figure}[h]
\includegraphics[width=0.5\linewidth]{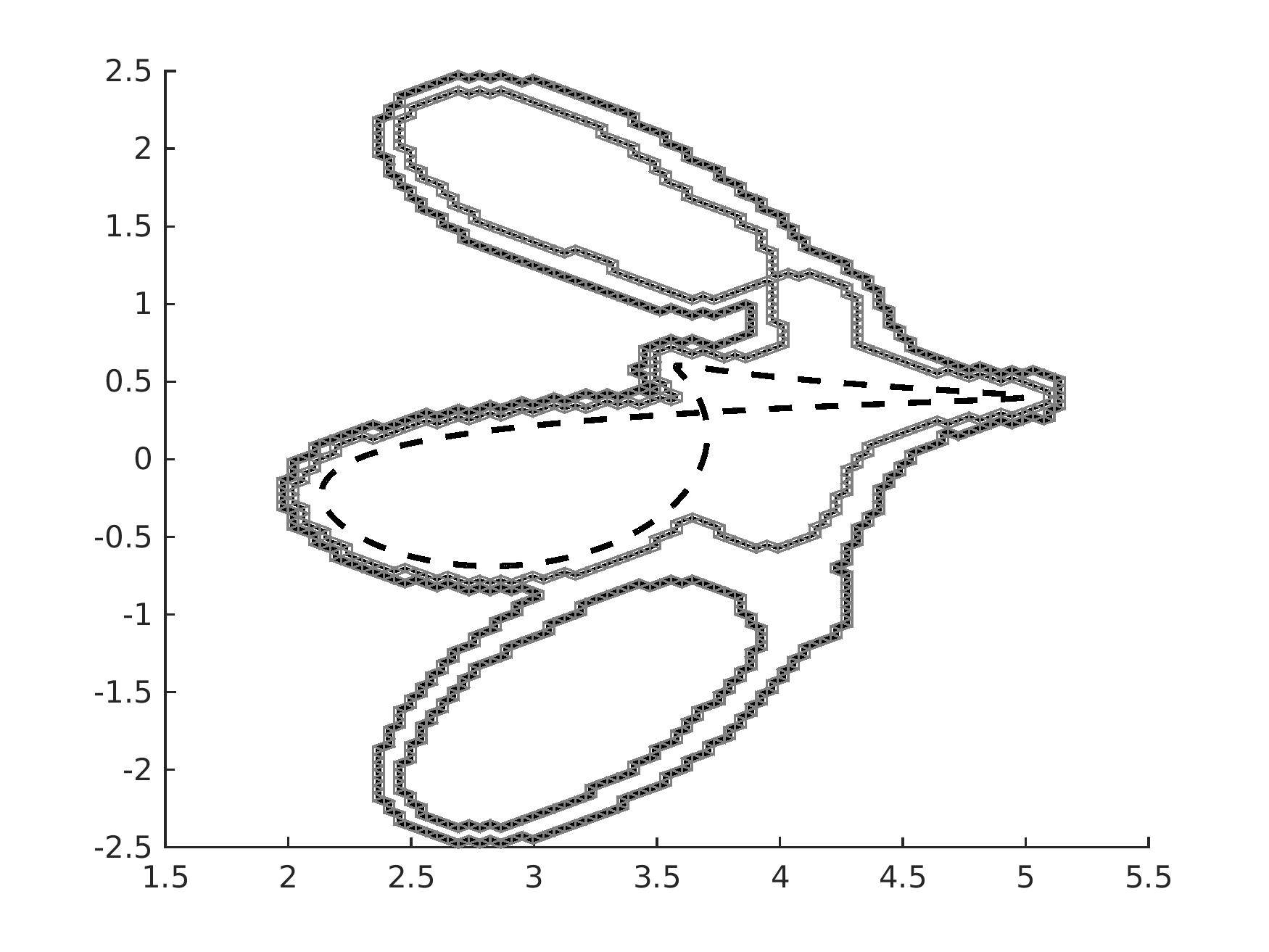}
\includegraphics[width=0.5\linewidth]{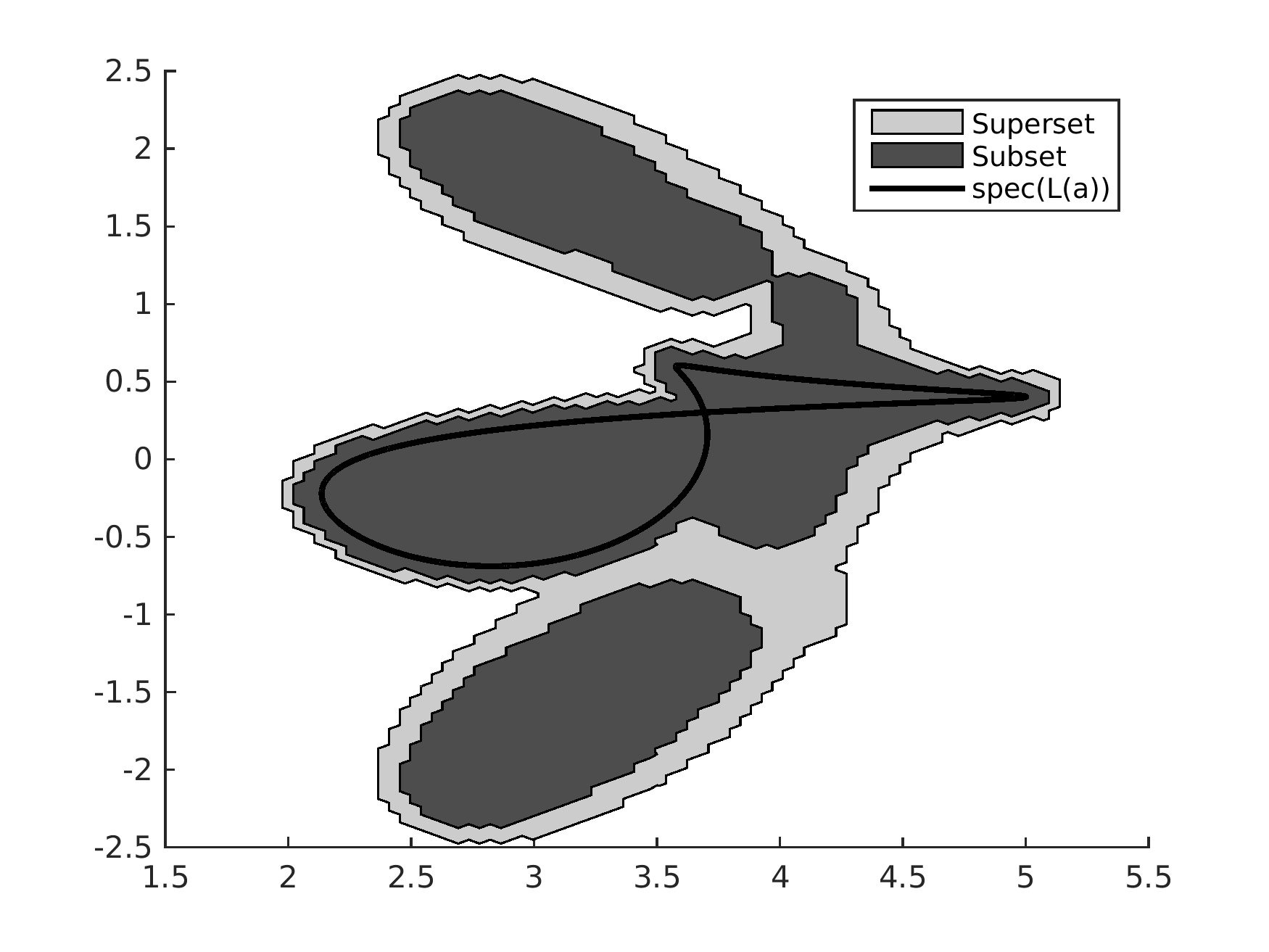}
\caption{Boundary sets for the impure Laurent operator with symbol \eqref{eq:impureLaurentsymbol} and a scaled $10\times 10$ Grcar-matrix as local impurity. 
Left: Triangulation of solution curves of \eqref{eq:boundaryFuncl} and \eqref{eq:boundaryFuncu} and $\spec(L(a))$ superimposed as a dotted line; 
Right: Visualization of the resulting sub- and superset of $\speps(L(a)+E)$. Dimensions are $d=15$, $N=200$ and $\eps=0.1$, and the approximation errors are $\eta_d\leq 1.2\cdot10^{-4}$ and $\delta_N\leq0.0512$. We used a total of $1486$ equilateral triangles of side-length $0.05$ for the computations.}
\label{fig:impureLautent}
\end{figure}
\begin{figure}[h]
\includegraphics[width=0.5\linewidth]{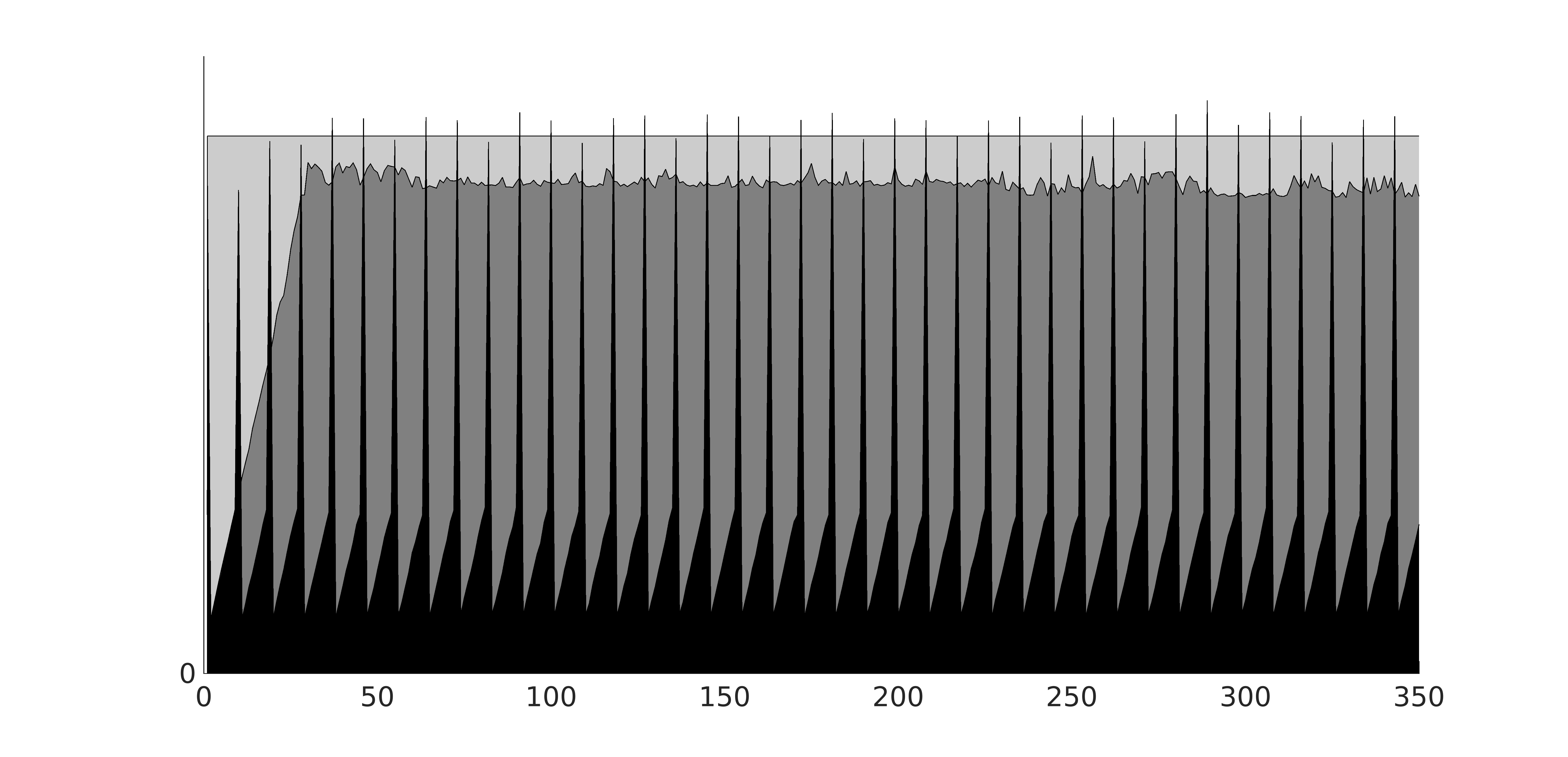}
\includegraphics[width=0.5\linewidth]{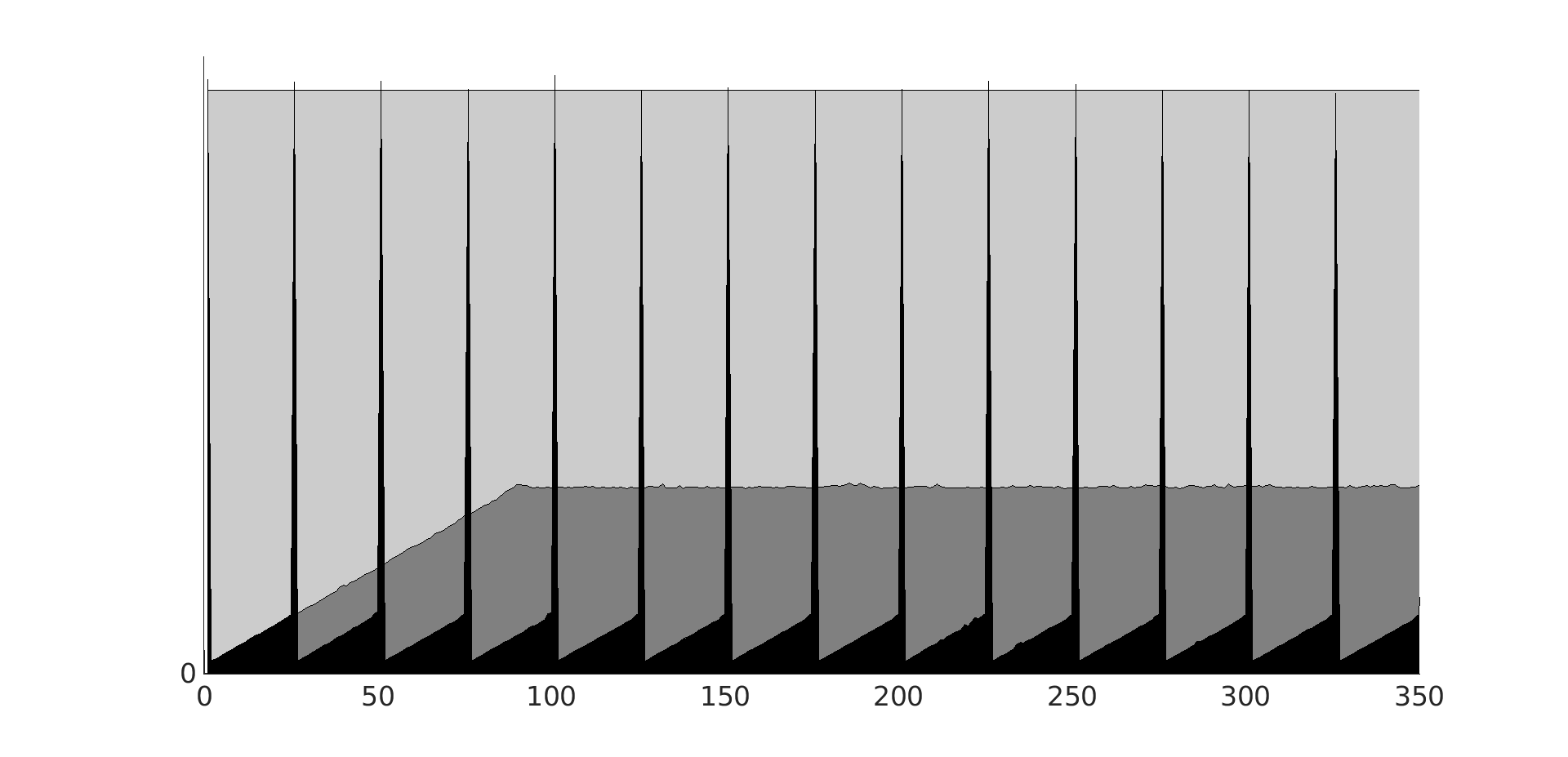}
\caption{Two plots of the computing time required for each step in Algorithm \ref{alg:qhshift} using the QH-Shift method (dark grey), 
the restarted QH-shift method (black) and, for comparison, the classical full QR decomposition using Givens rotations (light grey). 
The total computational cost corresponds to the total black, dark or light grey area. Comparison of the left ($d=15$) and right ($d=45$) 
image confirms that our algorithm pays off with increasing bandwidth. The cut-size is $N=50$ in both cases, so that $n=15\cdot 50=750$ 
and $n=45\cdot 50=2250$, respectively.}
\label{fig:impureLaurent2}
\end{figure}

\subsection{Singular Integral Operators}
Let $c$ and $e$ be continuous functions on $\T$ which define Laurent operators $L(c)$ and $L(e)$ on $\ell^2$ via the Fourier isomorphism (see Section \ref{ssec:laurentOperatorsImpurities} and \cite{BoeSi2}). In our numerical example below, we use the ``whale'' symbol from \cite{BoeSi2} and our ``fish'' symbol from \eqref{eq:impureLaurentsymbol}. Now we interbreed ``whale'' and ``fish'':

Put $S_{\ell^2}:=P-Q$ on $\ell^2$, where $P$ is the orthogonal projection of $\ell^2(\Z)$ onto $\ell^2(\N_0)$ and $Q:=I-P$. Then $S_{\ell^2}$ corresponds to the so-called Cauchy singular integral operator on $L^2(\T)$ (see e.g. \cite[p.130f]{LiBook}). After composition and addition with our multiplication operators by $c$ and $e$ on $L^2(\T)$, we study the bounded linear integral operator
\begin{align}
 Ax(t):=c(t)x(t)+\frac{e(t)}{\pi i}\int_{\T}\frac{x(s)}{s-t}ds,\;\;t\in\T\label{eq:singint}
\end{align}
on $L^2(\T)$, where the integral has to be understood in the sense of the Cauchy principal value. $A$ from \eqref{eq:singint} can be identified with $A_{\ell^2}$ on $\ell^2(\Z)$, where
\begin{align*}
 A_{\ell^2}=&L(c)+L(e) S_{\ell^2}=L(c)(P+Q)+L(e)(P-Q)\\
 =&L(c+e)P+L(c-e)Q=L(a)P+L(b)Q\quad\textrm{with}\quad a:=c+e,\ b:=c-e.
\end{align*}
The functions $a,b$ are continuous on $\T$ and, analogously to \eqref{eq:laurentsymbol}, the matrix of $A_{\ell^2}$ is of the form
\begin{align*}
A_{\ell^2}=\left(
\begin{array}{cccc|cccc}
\ddots&\ddots&\ddots&\ddots&\ddots&\ddots&\ddots&\ddots\\
\ddots&a_0&a_{-1}&a_{-2}&b_{-3}&b_{-4}&b_{-5}&\ddots\\
\ddots&a_1&a_0&a_{-1}&b_{-2}&b_{-3}&b_{-4}&\ddots\\
\ddots&a_2&a_1&a_0&b_{-1}&b_{-2}&b_{-3}&\ddots\\
\ddots&a_3&a_2&a_1&b_0&b_{-1}&b_{-2}&\ddots\\
\ddots&a_4&a_3&a_2&b_1&b_0&b_{-1}&\ddots\\
\ddots&\ddots&\ddots&\ddots&\ddots&\ddots&\ddots&\ddots
\end{array}\right).
\end{align*}
As there are only $N+1$ distinct submatrices formed by $N$ consecutive columns of $A_{\ell^2}$, our algorithm can be efficiently applied (i.e. only $N+1$ positions $k$ have to be considered; actually some more positions are needed for the adjoint) -- similarly to the situation in Section \ref{ssec:laurentOperatorsImpurities}.

We again define functions $F_l$ and $F_u$ as in \eqref{eq:boundaryFuncl} and \eqref{eq:boundaryFuncu} respectively. This time we want to approximate $\speps(A_{\ell^2})$ for many values of $\eps$ and therefore use a simple grid-based approach. We determine a finite grid $G\subset\C$ in the complex plane and then apply the QH-Shift method twice for every $\lambda\in G$. The result has been illustrated in Figure \ref{fig:singularint}.
\begin{figure}[h]
\includegraphics[width=0.5\linewidth]{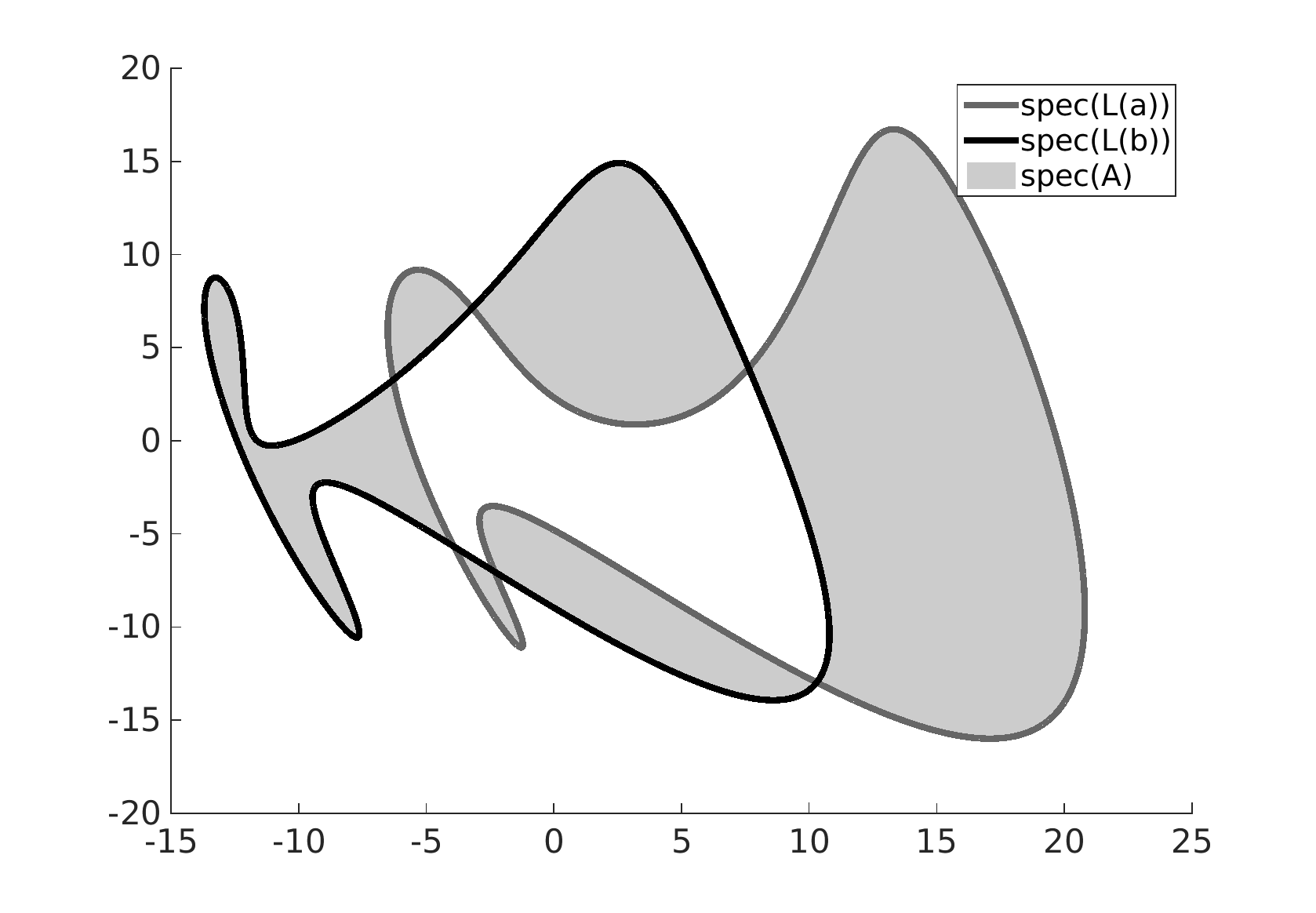}
\includegraphics[width=0.5\linewidth]{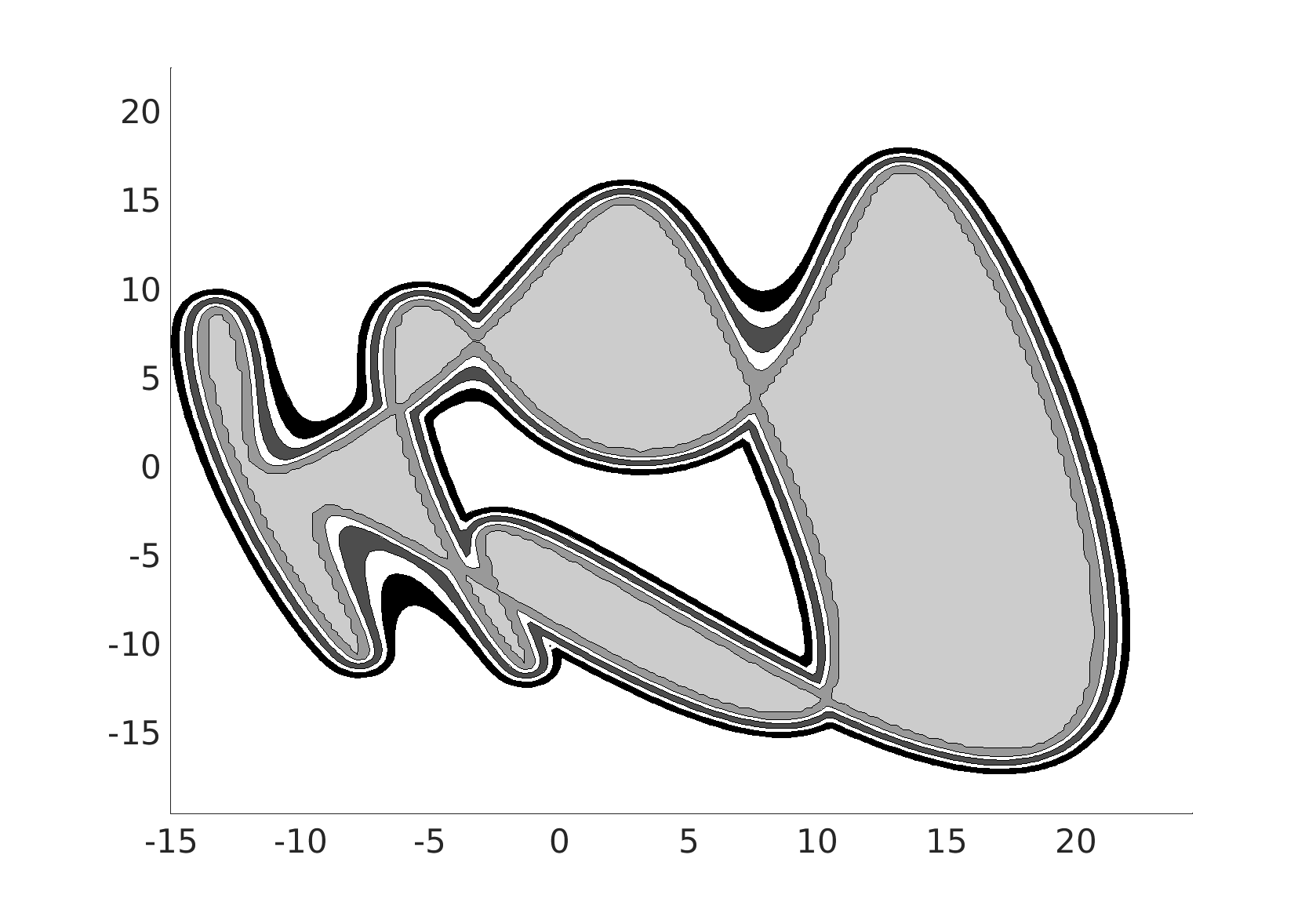}
\caption{Left: The spectrum of the Laurent operators and the resulting operator $A$. Right: Estimations of $\speps(A)$ for the three values $\eps=0.01,0.5,1.0$ (from light to dark grey) using $d=10$ and $N=200$. Each grey area contains the contour line of $\speps(A)$ for one of the three values of $\eps$. The higher we choose $N$, the smaller and therefore more accurate these areas become. The lightest grey area is contained in all three pseudospectra and is depicted here to clarify which parts of the plane (on which side of the contour lines) belong to the pseudospectra and which do not.}
\label{fig:singularint}
\end{figure}


\vfill
\noindent {\bf Authors' addresses:}\\
\\
Marko Lindner\hfill \href{mailto:lindner@tuhh.de}{\tt lindner@tuhh.de}\\
Torge Schmidt\hfill \href{mailto:torge.schmidt@tuhh.de}{\tt torge.schmidt@tuhh.de}\\
Institut Mathematik\\
TU Hamburg (TUHH)\\
D-21073 Hamburg\\
GERMANY

\end{document}